\documentclass[12pt,a4paper]{amsart}
\usepackage{verbatim}
\usepackage{setspace}
\usepackage[left=2.5cm,right=2.5cm,top=3.0cm,bottom=3.0cm,includeheadfoot]{geometry}
\usepackage{latexsym}
\usepackage{amssymb}
\usepackage{amsthm}
\usepackage{amsmath}
\usepackage{tikz}
\usepackage{enumitem}
\usetikzlibrary{matrix}
\usetikzlibrary{positioning}
\usetikzlibrary{arrows}
\usepackage{graphicx}
\usepackage[font=small,labelfont=bf]{caption}
\usepackage{pgfplots}
\usepackage{float}

\tikzset{
>=stealth',
punkt/.style={
           rectangle,
           rounded corners,
           draw=black, very thick,
           text width=6.5em,
           minimum height=2em,
           text centered},
pil/.style={
           ->,
           thick,
           shorten <=2pt,
           shorten >=2pt,}
}
\newcommand\numberthis{\addtocounter{equation}{1}\tag{\theequation}}

\usepackage[english]{babel}

\theoremstyle{plain}
\newtheorem{thm}{Theorem}
\newtheorem{theorem}[thm]{Theorem}
\newtheorem{corollary}[thm]{Corollary}
\newtheorem{lemma}[thm]{Lemma}

\newtheorem{proposition}[thm]{Proposition}

\newtheoremstyle{exm}
{9pt}{9pt}{}{}{\bfseries}{}{.5em}{}
\theoremstyle{exm}
\newtheorem{exm}[thm]{Example}
\newtheoremstyle{rmk}
{9pt}{9pt}{}{}{\bfseries}{}{.5em}{}
\theoremstyle{rmk}
\newtheorem{rmk}[thm]{Remark}

\theoremstyle{alg}

\newtheoremstyle{question}
{9pt}{9pt}{}{}{\bfseries}{}{.5em}{}
\theoremstyle{question}

\numberwithin{equation}{section}
\numberwithin{thm}{section}
\numberwithin{figure}{section}

\theoremstyle{definition}
\newtheorem{definition}[thm]{Definition}
\newtheorem{defin}[thm]{Definition}
\newtheorem{example}[thm]{Example}
\newtheorem{remark}[thm]{Remark}

\newtheorem{qn}{Question}

\newcommand{\field}[1]{\mathbb{#1}}
\newcommand{\C}{\field{C}}
\newcommand{\R}{\field{R}}

\newcommand{\Z}{\field{Z}}
\newcommand{\Q}{\field{Q}}

\usepackage[latin1]{inputenc}

\newcommand{\ind}{\operatorname{Ind}}

\newcommand{\Hi}{\mathrm{H}}

\def \I {\textrm{ Ind}}

\def \ssp {S \smallsetminus S_P}

\def \hp {\mathrm{H}_P}

\def \ob {G_{\mathbb{C}}/B}
\def \ll {{\mathbb L}}
\def \tn {\tilde{n}}
\def \tnb {\tilde{n}_{\beta}}
\def \gp {{G_{\mathbb{C}}/P}}

\DeclareMathOperator{\hh}{H}

\title[Generalizing the Mukai Conjecture]{Generalizing the Mukai Conjecture to the
symplectic category and the Kostant game}

\author[A. Caviedes Castro]{Alexander Caviedes Castro}
\address{Department Mathematik/Informatik, Universit\"at zu K\"oln,
  Weyertal 86-90, 50931 K\"oln, Germany.}
\email{alephxander@gmail.com}

\author[M. Pabiniak]{Milena Pabiniak}
\address{Department Mathematik/Informatik, Universit\"at zu K\"oln,
  Weyertal 86-90, 50931 K\"oln, Germany.}
\email{milenapabiniak@gmail.com}

\author[S. Sabatini]{Silvia Sabatini}
\address{Department Mathematik/Informatik, Universit\"at zu K\"oln,
  Weyertal 86-90, 50931 K\"oln, Germany.}
\email{sabatini@math.uni-koeln.de}

\date{\today}

\begin{document}

\maketitle

\begin{abstract}
In this paper we pose the question of whether the (generalized) Mukai inequalities hold for compact, positive monotone symplectic manifolds.
We first provide a method that enables one to check whether the (generalized) Mukai inequalities hold true. This only makes use of the almost complex structure
of the manifold
and the analysis of the zeros of the so-called generalized Hilbert polynomial, which takes into account the Atiyah-Singer indices of all possible line bundles.

We apply this method to generalized flag varieties. In order to find the zeros of the corresponding generalized Hilbert polynomial we introduce a modified version of the Kostant game and study its combinatorial properties.
\end{abstract}


\section{Introduction}

In 1988 Mukai \cite{mukai} conjectured that a Fano variety $M$ of complex dimension $n$ with index $k_0$ and Picard number $b$ should satisfy the following inequality:
$$
n \geq b(k_0-1)\,,
$$
with equality if and only if $M$ is $(\C P^{k_0-1})^b$. This conjecture was generalized by Bonavero, Casagrande, Debarre and Druel in \cite{bcdd}, where the
index above is replaced by the pseudoindex $\rho_M$, defined as the minimum of the evaluation of the anticanonical divisor $-K_M$ on rational curves on $M$.
There has been extensive work towards proving these two inequalities in a variety of cases, but general proofs of these conjectures are still missing.

In this paper we start investigating similar questions in a different category. Namely, suppose that $(M,\omega)$ is a compact symplectic manifold with first Chern class of the tangent bundle given by $c_1$. The \textit{index} of $(M,\omega)$ can be defined as the largest integer $k_0$ satisfying $c_1=k_0 \eta$ for some primitive element $\eta\in H^2(M;\Z)$,
modulo torsion elements. The symplectic analogues of Fano varieties are known in the literature as \textit{positive monotone symplectic manifolds}, namely symplectic manifolds
satisfying $c_1=[\omega]$. As for Fano varieties the Picard number is exactly the second Betti number, the following question arises naturally:
\begin{qn}
Let $(M,\omega)$ be a compact, positive monotone symplectic manifold of dimension $2n$ with second Betti number $b_2$ and index $k_0$. Does the following inequality hold?
\begin{equation}\label{mukai}
n\geq b_2(k_0-1)\quad\quad\quad\text{\textbf{Symplectic Mukai inequality}}
\end{equation}
\end{qn}
For positive monotone symplectic manifolds the concept of pseudoindex can also be generalized, namely:
Consider an embedded symplectic sphere $S^2$ and observe that, under the positive monotonicity assumption,
$c_1[S^2]$ is a positive integer. Therefore one can define the \textit{pseudoindex} $\rho_0$ of $(M,\omega)$ as
$$
\rho_0:=\min\{c_1[S^2]\mid S^2 \text{ symplectic sphere embedded in }(M,\omega)\}.
$$
In analogy with the generalized Mukai conjecture for Fano varieties we pose the following:
\begin{qn}\label{gen mukai}
Let $(M,\omega)$ be a compact, positive monotone symplectic manifold of dimension $2n$ with second Betti number $b_2$ and pseudoindex $\rho_0$. Does the following inequality hold?
\begin{equation}\label{gen mukai ine}
n\geq b_2(\rho_0-1)\quad\quad\quad \text{\textbf{Generalized symplectic Mukai inequality}}
\end{equation}
\end{qn}
One of the reasons to believe that such inequalities should hold true for positive monotone symplectic manifolds is that recent work has shown that,
at least under some mild symmetry assumptions, this category behaves very similarly to that of Fano varieties, see for instance \cite{fp, lp, ss}.

As Questions 1 and 2 have been so far investigated only for Fano varieties, many of the tools used to answer them are, not surprisingly, algebraic geometric.
Therefore, the first step to tackle them is to generalize the methods with which they can be proved to a category of spaces that includes that of positive monotone symplectic manifolds.

In Section \ref{sec:zeros} we introduce the so-called \emph{generalized Hilbert polynomial} $\Hi$ for a compact almost complex manifold $(M,J)$. This polynomial
takes into account the Atiyah-Singer indices of all the possible line bundles on $(M,J)$. One of the main results of the section is Theorem \ref{theorem hilbert polynomial},
in which it is proved that the presence of a so-called ``pointed box of zeros'' of $\Hi$ (see Definition \ref{pointed box}) implies an inequality that is related to the
Mukai inequality and its generalization. This relation is investigated in Corollary \ref{corollary hilbert polynomial} and Corollary \ref{GKM monotone}.
In particular the latter gives conditions under which a positive monotone Hamiltonian GKM space satisfies an inequality that is indeed stronger than the generalized Mukai
inequality (see \eqref{bound degree H3}).

One of the goals of Section \ref{mukai and kostant} is to verify the
hypotheses of Corollary \ref{GKM monotone} for generalized flag
varieties, thus proving for them the existence of a ``pointed box of
zeros'' of the generalized Hilbert polynomial $\Hi$, see Theorem
\ref{thm strong}. We point out that the corresponding inequality
\eqref{strong inequality} in Corollary \ref{inequality pasquier} has
already been proved by Pasquier \cite{pasquier}, where he checked it
through explicit calculations that depend on the type of the Dynkin
diagram of the compact Lie group and the parabolic subgroup. The
proof that we present in this paper relies on the combinatorics of
the \textit{(modified) Kostant game}, see Subsections
\ref{sc:kostant game} and \ref{sc:modified kostant game}. One of the
main purposes of the already known Kostant game (see for instance
\cite{postnikov, elek}) is to enumerate the positive roots of a
compact Lie group from its Dynkin diagram. In Subsection
\ref{sc:modified kostant game} we explain a modified version of the
Kostant game, which allows us to index the linear factors of $\Hi$,
thus finding its zeros.\footnote{The referee let us know that the
modified version of the Kostant game already appeared in
\cite{nakajima}.} It turns out that the moves of the modified
Kostant game are in one to one correspondence with the reduced
expressions of the minimal length representatives in the posets of
the quotient of a Weyl group with a parabolic subgroup. This fact is
purely combinatorial and can be generalized to any Weyl group and
any of its parabolic subgroups.

\subsection*{Acknowledgments} 
The authors would like to thank the referee for its careful reading
of the paper and the suggestions that improved the exposition.
The authors were partially supported by
SFB-TRR 191 grant {\it Symplectic Structures in Geometry, Algebra and
  Dynamics} funded by the Deutsche Forschungsgemeinschaft.

\section{Zeros of the generalized Hilbert polynomial and Mukai inequalities}\label{sec:zeros}
Let $(M,J)$ be a compact, connected, almost complex manifold of dimension $2n$. Consider the cohomology group
$H^2(M;\Z)$ and the following map
\begin{center}
\begin{tabular}{r  r c l }
$\widetilde{H} \colon$ & $H^2(M;\Z)$ & $\longrightarrow$ & $\Z$   \\
                                & $\eta$ & $\mapsto$  & $\ind (e^{2\pi \mathrm{i}\eta})$,
\end{tabular}
\end{center}
where  the first Chern class of the line bundle $e^{2\pi
\mathrm{i}\eta}$ is exactly $\eta$. We refer to it as the
\emph{index map}, as it computes the indices of all the possible
line bundles on $(M,J)$. Thanks to a simple observation (see for instance \cite[Lemma 4.1]{sabatini}), the
map $\widetilde{H}$ is well-defined on the lattice $\mathcal{L}$
given by the quotient $H^2(M;\Z)/\operatorname{Tor}(H^2(M;\Z))$,
where $\operatorname{Tor}(H^2(M;\Z))$ denotes the torsion subgroup
of $H^2(M;\Z)$. By abuse of notation, let $\widetilde{H}\colon
\mathcal{L}\to \Z$ be the induced index map.

Let $b_2$ denote the second Betti number of $M$, which is therefore exactly the rank of $\mathcal{L}$. Once a $\Z$-basis $\tau_1,\ldots,\tau_{b_2}$
of $\mathcal{L}$ is chosen, we can identify $\mathcal{L}$ with $\Z^{b_2}$
\begin{center}
\begin{tabular}{c l c }
$\mathcal{L}$ & $\longrightarrow$ & $\Z^{b_2}$   \\
$\eta=k_1 \tau_1 + \cdots + k_{b_2}\tau_{b_2}$ & $\mapsto$ & $(k_1,\ldots,k_{b_2})$
\end{tabular}
\end{center}
and obtain a map $\mathrm{H}\colon \Z^{b_2}\to \Z$ given by
\begin{equation}\label{def hilbert}
\mathrm{H}(k_1,\ldots,k_{b_2})= \ind (e^{2\pi \mathrm{i}\eta}),\quad\text{where}\quad \eta=k_1\tau_1+\cdots + k_{b_2}\tau_{b_2}.
\end{equation}
\begin{lemma}\label{lemma hilbert polynomial}
The map $\mathrm{H}\colon \Z^{b_2}\to \Z$ defined in \eqref{def hilbert} is a polynomial in $k_1,\ldots,k_{b_2}$ of degree at most $n$.
\end{lemma}
\begin{proof}
This lemma is a direct consequence of the Atiyah-Singer cohomological formula for the index \cite{AS}, as for $\eta=k_1\tau_1+\cdots k_{b_2}\tau_{b_2}$, the index of
$e^{2\pi \mathrm{i}\eta}$ is given by
\begin{align*}
\ind(e^{2\pi \mathrm{i}\eta}) & = \text{Ch}(e^{2\pi \mathrm{i}\eta})\mathcal{T}[M]=\text{Ch}(e^{2\pi \mathrm{i}(k_1\tau_1+\cdots + k_{b_2}\tau_{b_2})})\mathcal{T}[M]\\
                                            & = \left( \sum_{l\geq 0} \frac{(k_1\tau_1+\cdots + k_{b_2}\tau_{b_2})^l}{l!}\right)\mathcal{T}[M]\\
                                            & = \sum_{I}a_I k_1^{l_1}\cdot\,\cdots\,\cdot k_{b_2}^{l_{b_2}}\,,
\end{align*}
where the sum runs over all the multi-indices $I=(l_1,\ldots,l_{b_2})$ satisfying $l_j\in \Z_{\geq 0}$ for all $j$ and $l_1+\cdots + l_{b_2}\leq n$, $\mathcal{T}$
is the total Todd class of $M$ and $[M]$ is the orientation class of $(M,J)$ in homology.
\end{proof}
Observe that the coefficients $a_I$ above are rational numbers given by rational combinations of the evaluation on the homology class $[M]$ of products of Chern classes of $(M,J)$ and of
the classes $\tau_1,\ldots,\tau_{b_2}$.

Lemma \ref{lemma hilbert polynomial} allows us to define the following polynomial on $\mathbb{C}^{b_2}$, which is exactly the (unique) polynomial extension of $\mathrm{H}$
from $\Z^{b_2}$ to $\C^{b_2}$ and, by abuse of notation, is still denoted by $\mathrm{H}$.
\begin{defin}\label{def hilbert polynomial}
The \textbf{generalized Hilbert polynomial} is the polynomial map
\begin{center}
\begin{tabular}{r  c c l }
$\mathrm{H} \colon$ & $\C^{b_2}$ & $\longrightarrow$ & $\C$   \\
                                & $(z_1,\ldots,z_{b_2})$ & $\mapsto$  & $ \sum_{I}a_I z_1^{l_1}\cdot\,\cdots\,\cdot z_{b_2}^{l_{b_2}}$,
\end{tabular}
\end{center}
where the sum runs over all the multi-indices $I=(l_1,\ldots,l_{b_2})$ satisfying $l_j\in \Z_{\geq 0}$ for all $j$ and $l_1+\cdots + l_{b_2}\leq n$, and the $a_I$'s are defined in the
proof of Lemma \ref{lemma hilbert polynomial}.
\end{defin}

\begin{exm}
Let $(M,J)$ be a compact almost complex manifold of dimension 4 with Chern classes of $(TM,J)$ given by $c_1$ and $c_2$. The total Todd class is given
in this case by $\mathcal{T}=1+\frac{c_1}{2}+\frac{c_1^2+c_2}{12}$.
Assume that $\mathcal{L}$ has dimension 2 and let
$\tau_1,\tau_2$ be one of its bases. Let $c_1=\alpha_1\tau_1 + \alpha_2 \tau_2$, where $\alpha_1,\alpha_2\in \Z$.
Then
\begin{align*}
\mathrm{H}(k_1,k_2) &= e^{2\pi \mathrm{i} (k_1 \tau_1 + k_2 \tau_2)}\mathcal{T}[M] = \left( 1+k_1\tau_1 + k_2 \tau_2 + \frac{(k_1\tau_1+k_2\tau_2)^2}{2}\right)
(1+\frac{c_1}{2}+\frac{c_1^2+c_2}{12})[M]\\
& = k_1^2 \frac{\tau_1^2}{2}[M]+ k_2^2 \frac{\tau_2^2}{2}+ k_1 k_2 \tau_1\tau_2[M]+ k_1\frac{\alpha_1\tau_1^2+\alpha_2 \tau_1\tau_2}{2}[M]+ k_2 \frac{\alpha_1 \tau_1 \tau_2 + \alpha_2 \tau_2^2}{2}[M] \\
& + \mathrm{Todd}(M),
\end{align*}
where $\mathrm{Todd}(M)$ denotes the Todd genus of $M$ and is given by $\frac{c_1^2+c_2}{12}[M]$. In the notation introduced in Lemma \ref{lemma hilbert polynomial} we obtain $a_{(2,0)}= \frac{\tau_1^2}{2}[M]$, $a_{(0,2)}=\frac{\tau_2^2}{2}$, $a_{(1,1)}=\tau_1\tau_2[M]$ and so on.
\end{exm}

The following theorem is one of the main results of this section.

\begin{theorem}\label{theorem hilbert polynomial}
Let $(M, J)$ be a compact, connected, almost complex manifold of
dimension $2n$ with second Betti number $b_2$ and first Chern class $c_1$.
Let $\tilde{\mathcal{L}}$ be a full rank
sublattice of $\mathcal{L}=H^2(M;\Z)/\operatorname{Tor}(H^2(M;\Z))$
and $\{\eta_1, \ldots, \eta_{b_2}\}$ be a  $\Z$-basis of
$\tilde{\mathcal{L}}$ such that $c_1\in \tilde{\mathcal{L}}$ and
$$
c_1=\sum_{i=1}^{b_2}m_i\eta_i
$$
for some positive integers $m_i\in \Z_{>0}\,.$

Let $\widetilde{H}\colon \mathcal{L}\to \Z$ be the index map and assume that $$\widetilde{H}(-c_1)\neq 0,$$ and
\begin{equation}\label{zeros Hil}
\widetilde{H}(-k_1\eta_1- \cdots -k_{b_2}\eta_{b_2})=0,
\end{equation}
for all $k_1,\ldots,k_{b_2}\in \Z$ such that $0<k_i\leq m_i$ for all $i=1,\ldots,b_2$ and
$\sum_{i=1}^{b_2}k_i<\sum_{i=1}^{b_2}m_i.$
Let $\Hi\in \Q[z_1,\ldots,z_{b_2}]$ be the generalised Hilbert polynomial.
Then
\begin{equation}\label{bound degree H}
n\geq \operatorname{deg}(\mathrm{H})\geq
{\sum_{i=1}^{b_2}(m_i-1)}\,.
\end{equation}
\end{theorem}

\begin{rmk}
It is easy to see that the first inequality in  \eqref{bound degree H} is an equality whenever $c_1^n[M]\neq 0$. This holds for instance
when $M$ is a complex Fano variety.
\end{rmk}

In order to have a better understanding of this theorem we introduce the following terminology.
\begin{definition}\label{pointed box}
For $m=(m_1, \ldots, m_l)\in \Z_{\geq 0}^l$, we call the set
$$\{(k_1, \ldots, k_l)\in \Z^l_{\geq 0}\mid (k_1, \ldots, k_l)\leq (m_1,
\ldots, m_l) \text{ and }(k_1,\ldots,k_l)\neq (m_1,\ldots,m_l)\}$$
 a \textit{pointed box at $m$}\,. An affine
transformation of a pointed box is also called \textit{pointed
box}\,.
\end{definition}
With this definition at hand, Theorem \ref{theorem hilbert
polynomial} means that the line bundles corresponding to the integral
points of $\tilde{\mathcal{L}}$ inside the ``pointed box'' of the figure sketched below, namely the white circles,
have index zero.

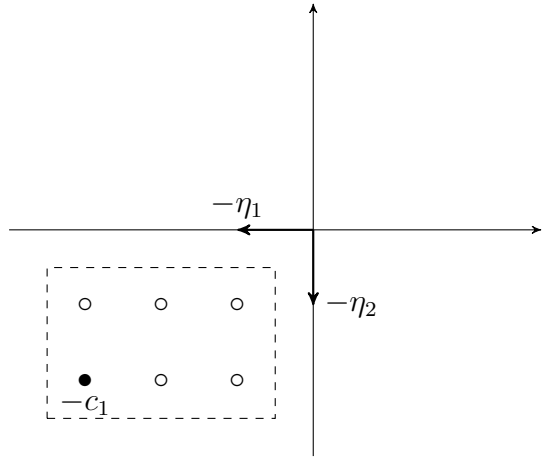
\begin{figure}[!ht]
\centering
\begin{tikzpicture}
\draw [->] (0,-3) -- (0,3); \draw [->] (-4,0) -- (3,0); \draw
[thick,->] (0,0) -- (-1,0) node[above] {$-\eta_1$}; \draw [thick,->]
(0,0) -- (0,-1) node[right] {$-\eta_2$}; \node at (-3,-2)
{\textbullet}; \node[below] at (-3,-2) {$-c_1$}; \foreach \Point in
{(-1,-1), (-2,-1),(-3,-1), (-1,-2), (-2,-2)}{
    \node at \Point {$\circ$};
} \draw[thin, dashed] (-0.5, -0.5) rectangle (-3.5,-2.5);
\end{tikzpicture}
\end{figure}

The proof of Theorem \ref{theorem hilbert polynomial} relies on the following algebraic fact.
First of all we set the following definitions.

\begin{itemize}
\item Let $\leq$ be the partial order on $\Z^l$, defined as follows:
$$
(s_1, \ldots, s_l)\leq (k_1, \ldots, k_l)\,\, \text{if and only
if}\,\, s_1\leq k_1, \ldots, s_l\leq k_l\,.
$$
\item Given the standard order on $\Z$, we define the following order preserving function
\begin{align*}
\operatorname{ht}:\Z^l&\to \Z\\
(k_1, \ldots, k_l) &\mapsto \operatorname{ht}(k_1, \ldots,
k_l):=k_1+\ldots+k_l
\end{align*}
and we call it \textit{height}\,.
\end{itemize}

\begin{lemma}
Let $P\in \Q[z_1, \ldots, z_l]$ and assume that there exists $(m_1,
\ldots, m_l)\in \Z_{\geq 0}$ such that $P(m_1, \ldots, m_l)\ne 0$
and $P(k_1, \ldots, k_l)=0$ for every $(k_1, \ldots, k_l)\in
\Z^l_{\geq 0}$ such that $(k_1, \ldots, k_l)\leq (m_1, \ldots, m_l)$ and $(k_1, \ldots, k_l)\neq (m_1, \ldots, m_l)$.
Then
\begin{equation}\label{claim lemma}
\operatorname{deg}P\geq \operatorname{ht}(m_1, \ldots,
m_l)=m_1+\ldots+m_l\,.
\end{equation}
\end{lemma}
\begin{proof}
The idea of the proof is to write the polynomial $P$ in terms of an
appropriate basis, show that one of the coefficients of $P$ with
respect to this basis is non zero and that the degree of the
corresponding basis element equals $\operatorname{ht}(m_1, \ldots,
m_l)\,.$

For $s \in \Z_{\geq 0},$ let
$$
M_{s+1}(z):=1\cdot z\cdot (z-1)\cdot \ldots \cdot
(z-s)=\prod_{j=0}^s(z-j)
$$
and $M_0(z):=1\,.$ The set $\{M_s(z)\}_{s\in \mathbb{Z}_{\geq 0}}$
defines a basis of $\Q[z]\,.$ Note that for an integer $k\in
\Z_{\geq 0}\,$
\begin{equation}\label{cases M}
M_s(k)=
\begin{cases}
k! & \text{if }k=s\\
0 & \text{if }k<s\,.
\end{cases}
\end{equation}
It can be easily checked that a basis of $\Q[z_1, \ldots, z_l]$ is the set
$$
\{M_{(s_1, \ldots, s_l)}(z_1, \ldots,
z_l):=M_{s_1}(z_1)\cdot\ldots\cdot M_{s_l}(z_l)\}_{(s_1, \ldots,
s_l)\in \mathbb{Z}_{\geq 0}^l}\,.
$$
Therefore we can write any polynomial $P\in \Q[z_1, \ldots, z_l]$ as a linear
combination of the elements of this basis
\begin{equation}\label{P in basis}
P(z_1, \ldots, z_l)=\sum_{(s_1, \ldots, s_l)\in \mathbb{Z}^l_{\geq
0}}h_{(s_1, \ldots, s_l)}\cdot M_{(s_1, \ldots, s_l)}(z_1, \ldots,
z_l)
\end{equation}
for some rational numbers $h_{(s_1, \ldots, s_l)}$, such that all except
a finite number of them are equal to zero.

We show the following\\$\;$\\
\emph{Claim}: For a polynomial $P$ satisfying the assumptions of the
Lemma, the coefficient $h_{(k_1, \ldots, k_l)}$ equals zero for
every $(0, \ldots, 0)\leq (k_1, \ldots, k_l)\leq (m_1, \ldots, m_l)$ and $(k_1, \ldots, k_l)\neq (m_1, \ldots, m_l)$.
Moreover the coefficient $h_{(m_1, \ldots, m_l)}$ is nonzero.
\\$\;$\\
The claim implies that
$$\deg(P)\geq \deg(M_{(m_1, \ldots, m_l)})=m_1+\cdots+m_l\,,$$
which is the desired inequality \eqref{claim lemma}.

From \eqref{cases M}, the definition of $M_{(s_1,\ldots,s_l)}$ and \eqref{P in basis} it follows that
\begin{equation}\label{P in basis 2}
P(k_1, \ldots, k_l)=\sum_{(0, \ldots, 0)\leq (s_1, \ldots,
s_l)\leq(k_1, \ldots, k_l)}h_{(s_1, \ldots, s_l)}\cdot M_{(s_1,
\ldots, s_l)}(k_1, \ldots, k_l)
\end{equation}
for any $(k_1, \ldots, k_l)\in \Z^l_{\geq 0}\,.$ In particular
$P(0)=0$, unless $(m_1, \ldots,m_l)=(0, \ldots, 0)$, in which case we are done.
We prove the claim
using finite induction on the height function: assume that $h_{(k_1,
\ldots, k_l)}=0$ for all $(k_1, \ldots, k_l)\in \Z^l_{\geq 0}$ with
$(k_1, \ldots, k_l)\leq (m_1, \ldots, m_l)$ and
$\operatorname{ht}(k_1, \ldots, k_l)\leq k<\operatorname{ht}(m_1,
\ldots, m_l)\,.$ We show that for any $(k_1, \ldots,
k_l)\in \Z^l_{\geq 0}$ with $(k_1, \ldots, k_l)\leq (m_1, \ldots,
m_l)$ and $\operatorname{ht}(k_1, \ldots, k_l)=k+1$, either $(k_1,
\ldots, k_l)\neq (m_1, \ldots, m_l)$ and $h_{(k_1, \ldots, k_l)}=0$ or
$(k_1, \ldots, k_l)=(m_1, \ldots, m_l)$ and $h_{(m_1, \ldots,
m_l)}\ne 0\,.$ The induction on the height stops when we reach
$(k_1, \ldots, k_l)=(m_1, \ldots, m_l)$, which is the only maximal
element in the set $\{(s_1, \ldots, s_l)\in \Z^l_{\geq 0}:(s_1,
\ldots, s_l):(s_1, \ldots, s_l)\leq (m_1, \ldots, m_l)\}$ with
respect to the order that we defined above.

Let $(k_1,\ldots,k_l)\in \Z_{\geq 0}$ be such that $\operatorname{ht}(k_1,\ldots,k_l)=k+1$. Then from
\eqref{P in basis 2} and the induction hypothesis we obtain
$$
P(k_1, \ldots, k_l)= h_{(k_1,\ldots,k_l)} M_{(k_1,\ldots,k_l)}(k_1,\ldots,k_l),
$$
as any
$(s_1,\ldots,s_l)$ satisfying $(0,
\ldots, 0)\leq (s_1, \ldots, s_l)\leq (k_1, \ldots, k_l)$ and $(s_1,\ldots,s_l)\neq (k_1,\ldots,k_l)$ has height
less than $k+1$, and therefore by induction satisfies $h_{(s_1,\ldots,s_l)}=0$.

If $(k_1, \ldots, k_l)\leq (m_1, \ldots, m_l)$ and $(k_1, \ldots, k_l)\neq (m_1, \ldots, m_l)$, then
$h_{(k_1, \ldots, k_l)}=0$ as $P(k_1, \ldots, k_l)=0$ by the assumption of the Lemma
and $M_{(k_1, \ldots, k_l)}(k_1, \ldots, k_l)=k_1!\cdot\ldots\cdot
k_l!\ne 0\,.$ Likewise, if $(k_1, \ldots, k_l)=(m_1, \ldots, m_l),$
then $P(m_1,\ldots,m_l)\neq 0$ implies $h_{(m_1, \ldots, m_l)}\ne 0$. This finishes the proof of the claim and hence
of the Lemma.

\end{proof}

%
%

\begin{proof}[Proof of Theorem \ref{theorem hilbert polynomial}]
The first inequality, $n\geq \operatorname{deg}\Hi$, comes from Lemma \ref{lemma hilbert polynomial} and the definition of generalized Hilbert polynomial.
In order to prove the second inequality we introduce the following:
In analogy with the definition of the generalised Hilbert polynomial $\Hi(z_1,\ldots,z_{b_2})$, we define the polynomial $\Hi'(z_1,\ldots,z_{b_2})$ that at integral values $(k_1,\ldots,k_{b_2})\in \Z^{b_2}$ satisfies
$$\Hi'(k_1,\ldots,k_{b_2}):=\widetilde{H}(k_1\eta_1+\cdots + k_{b_2} \eta_{b_2})\,.$$
We observe that, as
 $\tilde{\mathcal{L}}$
is a full rank sublattice of $\mathcal{L}$, $\Hi(z_1,\ldots,z_{b_2})$ and $\Hi'(z_1,\ldots,z_{b_2})$ differ only by a linear, invertible transformation of the coordinates $(z_1,\ldots,z_{b_2})$,
namely there exists a linear, invertible transformation $A\colon \C^{b_2}\simeq H^2(M;\C)\to \C^{b_2}\simeq H^2(M;\C)$ satisfying $A(\tau_j)=\eta_j$
for all $j=1,\ldots,b_2$, implying
$\Hi'(z_1,\ldots,z_{b_2})=\Hi(A(z_1,\ldots,z_{b_2}))$.
From the fact that $A$ is linear and invertible we have $\deg(\Hi')=\deg(\Hi)$.

Consider now the polynomial
$$
Q(z_1,\ldots,z_{b_2}):=\Hi'(-z_1-1,\ldots,-z_{b_2}-1)\,,
$$
whose degree is clearly the degree of $\Hi'$, and hence the degree of $\Hi$.

We recall that the first Chern class $c_1$ satisfies $c_1=\sum_{j=1}^{b_2} m_i \eta_i$ for some positive integers $m_1,\ldots,m_{b_2}$ and observe that the polynomial
$Q$ satisfies
$$
Q(m_1-1,\ldots,m_{b_2}-1)=\Hi'(-m_1,\ldots,-m_{b_2})=\widetilde{H}(-m_1\eta_1-\cdots - m_{b_2}\eta_{b_2}),
$$
the latter being nonzero by assumption. Moreover by \eqref{zeros Hil} we have
$
Q(k_1,\ldots,k_{b_2})=0
$ for all $(k_1,\ldots,k_{b_2})\in \Z^{b_2}_{\geq 0}$ such that $(k_1,\ldots,k_{b_2})\leq (m_1-1,\ldots,m_{b_2}-1)$ and $(k_1,\ldots,k_{b_2})\neq (m_1-1,\ldots,m_{b_2}-1)$. Therefore Lemma \ref{lemma hilbert polynomial} implies that
$$
\deg(\Hi)=\deg(Q)\geq \sum_{i=1}^{b_2}(m_i-1)\,,
$$
which concludes the proof.
\end{proof}

We are now ready to derive the corollaries of this section that concern the Mukai inequality.
Suppose that $\{\tau_1,\ldots,\tau_{b_2}\}$ is a basis of $\mathcal{L}$ such that
$c_1=\sum_{i=1}^{b_2} n_i \tau_i$ for some positive integers $n_i$, for all $i=1,\ldots,b_2$. Observe that the index $k_0$ of
$(M,J)$ satisfies $k_0=\gcd(n_1,\ldots,n_{b_2})$.

\begin{corollary}\label{corollary hilbert polynomial}
Let $(M, J)$ be a compact, connected, almost complex manifold of
dimension $2n$ with second Betti number $b_2$, first Chern class $c_1$ and index $k_0$.

Let $\{\tau_1,\ldots,\tau_{b_2}\}$ be a basis of $\mathcal{L}=H^2(M;\Z)/\operatorname{Tor}(H^2(M;\Z))$
such that $$c_1=\sum_{i=1}^{b_2} n_i \tau_i$$
for some positive integers $n_i\in \Z_{>0}\,.$ For $i=1,\ldots,b_2$ define $\displaystyle\eta_i:=\frac{n_i}{k_0}\tau_i\in
\mathcal{L}$  and assume that the index map
$\widetilde{H}\colon \mathcal{L}\to \Z$ satisfies
 $$\widetilde{H}(-c_1)\neq 0,$$ and
\begin{equation}\label{zeros Hil2}
\widetilde{H}(-k_1\eta_1- \cdots -k_{b_2}\eta_{b_2})=0,
\end{equation}
for all $k_1,\ldots,k_{b_2}\in \Z$ such that $0<k_i\leq k_0$ for all $i=1,\ldots,b_2$ and
$\sum_{i=1}^{b_2}k_i< k_0b_2$.

Let $\Hi\in \Q[z_1,\ldots,z_{b_2}]$ be the generalised Hilbert polynomial.
Then $(M,J)$ satisfies the Mukai inequality, more precisely
\begin{equation}\label{bound degree H2}
n\geq \operatorname{deg}(\mathrm{H})\geq b_2(k_0-1)\,.
\end{equation}
\end{corollary}
\begin{proof}
The first inequality is a consequence of Lemma \ref{lemma hilbert polynomial} and the definition of Hilbert polynomial. The second inequality
follows easily from Theorem \ref{theorem hilbert polynomial}, as the set $\{\eta_1,\ldots,\eta_{b_2}\}$
is $\Z$-independent and the lattice $\mathcal{L}'$ given by
$\Z\langle \eta_1, \ldots,\eta_{b_2}\rangle$ is a full rank sublattice of $\mathcal{L}$. Moreover $c_1\in \mathcal{L}'$ and it is given by
$c_1=\sum_{i=1}^{b_2}k_0 \eta_i$.
\end{proof}

The second corollary of Theorem \ref{theorem hilbert polynomial} concerns the generalized Mukai inequality for a positive monotone symplectic manifold with pseudoindex $\rho_0$.
Henceforth we focus on the following category of spaces, which is also that appearing in the next section, and which enables us to find a special
basis of $H^2(M;\Z)$ such that the coefficients of the first Chern class in this basis are symplectic volumes of some special embedded spheres.

Suppose that $(M,\omega)$ is a compact symplectic manifold endowed with a Hamiltonian action of a compact torus $T$: We recall that
the action of a compact torus $T$ on a symplectic manifold $(M,\omega)$ is called Hamiltonian if there exists a $T$-invariant map $\psi\colon M \to \operatorname{Lie}(T)^*$, called
\textbf{moment map}, satisfying
$$
d \langle \psi(\cdot), \xi \rangle = -\iota_{\xi^\#}\omega\,,
$$
where $\langle \cdot , \cdot \rangle$ denotes the natural pairing between ${Lie}(T)^*$ and ${Lie}(T)$, $\xi^\#$ is the vector field generated by the action
and $\iota_{\xi^\#}\cdot$ denotes the contraction operator. The function $\psi^{\xi}\colon M \to \R$
defined as $\psi^\xi(\cdot)=\langle \psi(\cdot), \xi \rangle$
is called the $\xi$-component of the moment map.
The triple $(M,\omega,\psi)$ is called a \textbf{(compact) Hamiltonian $T$-space}. In this paper we consider
Hamiltonian $T$-spaces whose fixed point set --denoted by $M^T$-- is discrete and hence, by compactness of $M$, finite.
Moreover we assume that $(M,\omega,\psi)$ is a \textbf{GKM (Goresky-Kottwitz-MacPherson) space}, namely,
for every codimension one subtorus $H$,
the connected components of the set of points fixed by $H$ is of dimension at most 2. This condition can also be rephrased as follows.
Consider the isotropy action of $T$  on the tangent space at a fixed point $p\in M^T$ and its corresponding weights, called \textbf{isotropy weights} of $p$.
Then the action is GKM if and only if for every $p\in M^T$ the weights of the isotropy action are pairwise linearly independent.
(See \cite{gkm} for the original reference or for instance \cite[Chapter 11]{gs-supersymmetry}.) Indeed the two-dimensional components mentioned above
are spheres corresponding exactly to the fixed points of the codimension one subtorus $\exp\{\xi \in \operatorname{Lie}(T)\mid \langle \alpha , \xi \rangle =0 \}$.
Note that, as they are fixed components of a subgroup of $T$, they are symplectic submanifolds of $M$.
The restriction of the $T$-action to any of those has exactly two fixed points $p,q\in M^T$, and the intersection properties of the set of these isotropy spheres
is encoded in a graph $(V,E_{GKM})$ called the \textbf{GKM graph}: the vertex set is exactly the fixed point set $M^T$, and every edge represents an
isotropy sphere.
Henceforth we restrict to Hamiltonian $T$-spaces whose action is GKM.

The basis of $H^2(M;\Z)$ that allows us to translate Theorem \ref{theorem hilbert polynomial} in terms of the pseudoindex comes from a well-known
basis in the equivariant cohomology ring of $(M,\omega,\psi)$, called the \textbf{canonical basis}. In order to introduce it we need the following terminology:
Consider a generic component $\psi^\xi$ of the moment map, where generic means that $\langle \alpha, \xi \rangle\neq 0$ for every isotropy weight $\alpha$ occurring at a fixed
point $p$. Define $\lambda_p$ to be the number of negative weights at $p$, namely the number of isotropy weights $\alpha$ of $p$ such that $\langle \alpha, \xi \rangle < 0$, and let $\Lambda_p^-$ be the product of these weights. We say that a generic component $\psi^\#$ is \textbf{index increasing} if $\lambda_p<\lambda_q$
for every edge connecting $p$ to $q$ in $E_{GKM}$ such that $\psi^\xi(p)< \psi^\psi(q)$.

Since the manifold is acted on upon a torus $T$, which we assume to have dimension $d$, it is useful to consider the equivariant cohomology ring with $\Z$-coefficients $H^*_T(M;\Z)$, which
has the structure of a $H^*_T(\{p\};\Z)$-module, where $\{p\}$ is simply a point. Once a basis $\{x_1,\ldots,x_d\}$ of the dual lattice of $\operatorname{Lie}(T)^*$
is chosen, the latter can be identified with the polynomial ring $\Z[x_1,\ldots,x_d]$. Notice that $\Lambda_p^-\in \Z[x_1,\ldots,x_d]$. More generally the restriction
$\tau(p)$ to a fixed point $p$ of a class $\tau\in H^*_T(M;\Z)$ lives in this polynomial ring.

The theorem below was proved by Guillemin and Zara \cite{gz} over the rationals and then extended to the integers by Goldin and Tolman \cite{gt}.
\begin{theorem}\label{canonical basis}
Let $(M,\omega,\psi)$ be a compact Hamiltonian $T$-space such that the $T$-action is GKM. Suppose that there exists a $\xi\in \operatorname{Lie}(T)$
such that $\psi^\xi$ is index increasing. Then for every $p\in M^T$ there exists a unique element $\tilde{\tau}_p\in H^{2\lambda_p}_T(M;\Z)$
satisfying the following properties:
\begin{itemize}
\item[(i)] $\tilde{\tau}_p(p)=\Lambda_p^-$;
\item[(ii)] $\tilde{\tau}_p(q)=0$ for all $q\in M^T\setminus\{p\}$ such that $\lambda_q\leq \lambda_p$.
\end{itemize}
Moreover the set $\{\tilde{\tau}_p\}_{p\in M^T}$ of these elements is a basis of $H_T^*(M;\Z)$ as a module over $\Z[x_1,\ldots,x_d]$ and is called
the \textbf{canonical basis} of $(M,\omega,\psi)$ w.r.t. the component $\psi^\xi$.
\end{theorem}
Consider the natural restriction $r\colon H^*_T(M;\Z)\to H^*(M;\Z)$. Another virtue of this basis is that it restricts to a basis of $H^*(M;\Z)$ (regarded as a
$\Z$-module). Therefore the elements $\{\tau_p:=r(\tilde{\tau}_p), \lambda_p=1\}$ form a basis of $H^2(M;\Z)$. By abuse of notation we refer to these elements
as the canonical basis of $H^2(M;\Z)$.

Before proving the main property of this canonical basis, which allows us to link Theorem \ref{theorem hilbert polynomial} to the generalized Mukai conjecture,
we observe the following:
\begin{itemize}
\item Since $(M,\omega)$ is compact and symplectic, $H^2(M;\Z)\neq 0$, therefore there must be fixed points $p$ with $\lambda_p=1$;
\item For every fixed point $p$ with $\lambda_p=1$ there exists a unique symplectic, invariant sphere containing the (unique) minimum $p_0$ of $\psi^\xi$
and $p$ as respectively ``south'' and ``north pole''. Let $S^2_1,\ldots,S^2_{b_2}$ be the collection of these spheres.
\end{itemize}
We are ready for the following
\begin{proposition}\label{coefficients volumes}
Let $(M,\omega,\psi)$ be a compact Hamiltonian $T$-space such that the $T$-action is GKM. Suppose that there exists a $\xi\in \operatorname{Lie}(T)$
such that $\psi^\xi$ is index increasing and consider the canonical basis $\{\tau_1,\ldots,\tau_{b_2}\}$ described above.

Let $\alpha\in H^2(M;\Z)$ and write it as
$\alpha=\sum_{i=1}^{b_2}m_i\tau_i$, where $m_i\in \Z$ for all $i=1,\ldots,b_2$. Let $S^2_1,\ldots,S^2_{b_2}$ be the collection of symplectic spheres
connecting $p_0$ to the fixed points $p$ with $\lambda_p=1$.
Then, modulo reordering the elements $\tau_i$,
$$
m_i= \int_{S^2_i}\alpha,\quad\text{for all }i=1,\ldots,b_2\,.
$$
\end{proposition}
\begin{proof}
By the Kirwan surjectivity theorem \cite{kirwan}, every cohomology class $\tau\in H^*(M;\Z)$ admits an equivariant extension $\tilde{\tau}\in H_T^*(M;\Z)$, namely
there exists $\tilde{\tau}$ such that $r(\tilde{\tau})=\tau$. This extension is not unique, but for degree 2 elements any two extensions differ by a degree one
polynomial in $\Z[x_1,\ldots,x_d]$ without constant term. Therefore, for the given $\alpha\in H^2(M;\Z)$ there exists $\tilde{\alpha}\in H^2_T(M;\Z)$ such that
\begin{equation}\label{alpha tilde}
\tilde{\alpha}=\sum_{i=1}^{b_2}m_i \tilde{\tau_i}+P,
\end{equation}
for some degree one polynomial $P\in \Z[x_1,\ldots,x_d]$ without constant term. By evaluating \eqref{alpha tilde} at the minimum $p_0$ of $\psi^\xi$
and using property (ii) in Theorem \ref{canonical basis} we obtain that $\tilde{\alpha}(p_0)=P$. Let $p_j$ be a point with $\lambda_{p_j}=1$.
Then, evaluating \eqref{alpha tilde} at $p_j$, using property (ii) in Theorem \ref{canonical basis} and the previous computation, we have
$$
\tilde{\alpha}(p_j)=m_j \tilde{\tau}_j(p_j)+\tilde{\alpha}(p_0)=m_j \alpha_j + \tilde{\alpha}(p_0)\,,
$$
where $\alpha_j$ is the unique negative weight at $p_j$. Therefore $m_j$, thought as a rational function in $x_1,\ldots,x_d$, is exactly given by
\begin{equation}\label{mj}
m_j=\frac{\tilde{\alpha}(p_j)-\tilde{\alpha}(p_0)}{\alpha_j}\,.
\end{equation}
Let $S^2_j$ be the unique invariant symplectic sphere stabilized pointwise by $\exp\{\xi\in \operatorname{Lie}(T)\mid \langle \alpha_j , \xi \rangle=0\}$ and containing
$p_0$ and $p_j$. Then the weight of the isotropy $T$-action on the tangent space $T_{p_0}S^2_{p_j}$ is exactly $-\alpha_j$. Therefore, by the Atiyah-Bott \cite{AB} and Berline-Vergne \cite{BV} localization
formula in equivariant cohomology, $m_j$ is exactly $\int_{S^2_j}\tilde{\alpha}$ which, by degree reasons, is equal to $\int_{S^2_j}\alpha$.
\end{proof}

\begin{corollary}\label{GKM monotone}
Let $(M,\omega)$ be a positive monotone compact symplectic manifold, therefore $c_1=[\omega]$, with pseudoindex $\rho_0$. Assume that
$(M,\omega)$ can be endowed with a Hamiltonian $T$-action which is also GKM. Suppose that there exists a $\xi\in \operatorname{Lie}(T)$
such that $\psi^\xi$ is index increasing and consider the canonical basis $\{\tau_1,\ldots,\tau_{b_2}\}\subset H^2(M;\Z)$ described above.
Let
$$
c_1=\sum_{i=1}^{b_2}m_i\tau_i
$$
for some integers $m_i\in \Z\,.$ Then $m_i>0$ for all $i=1,\ldots,b_2$.

Moreover, if $\widetilde{H}\colon \mathcal{L}\to \Z$ denotes the index map and
\begin{equation}\label{zeros Hil3}
\widetilde{H}(-k_1\tau_1- \cdots -k_{b_2}\tau_{b_2})=0,
\end{equation}
for all $k_1,\ldots,k_{b_2}\in \Z$ such that $0<k_i\leq m_i$ for all $i=1,\ldots,b_2$ and
$\sum_{i=1}^{b_2}k_i<\sum_{i=1}^{b_2}m_i,$ then
\begin{equation}\label{bound degree H3}
n\geq \sum_{i=1}^{b_2}(m_i-1)\geq b_2(\rho_0-1)\,,
\end{equation}
and therefore $(M,\omega)$ satisfies the generalized Mukai inequality.
\end{corollary}
\begin{proof}
Let $S^2_1,\ldots,S^2_{b_2}$ be the collection of symplectic spheres
connecting $p_0$, the minimum of $\psi^\xi$, to the fixed points $p_j$ with $\lambda_{p_j}=1$, for $j=1,\ldots,b_2$.
Then by Proposition \ref{coefficients volumes} and the monotonicity condition $c_1=[\omega]$ we have that
$$
m_j=\int_{S^2_j} c_1=\int_{S^2_j}\omega >0\,,
$$
thus proving the first claim.

Now we observe that for the class of spaces described in the hypotheses, $\widetilde{H}(-c_1)\neq 0$ (this holds indeed for all compact Hamiltonian $T$-spaces with discrete
fixed point set). Indeed, by \cite[Proposition 41]{sabatini}, $\widetilde{H}(-c_1)=(-1)^n N_0$, where $N_0$ is the number of fixed points with zero negative weights which,
for connected, compact Hamiltonian $T$-spaces is exactly 1. Therefore, the inequality
$$
n\geq \sum_{i=1}^{b_2}(m_i-1)
$$
in \eqref{bound degree H3} is a direct consequence of Theorem \ref{theorem hilbert polynomial}, where we can take $\eta_i=\tau_i$ for all $i=1,\ldots,b_2$.

The last inequality in \eqref{bound degree H3} follows from Proposition \ref{coefficients volumes}, because
$m_j=\int_{S^2_j} c_1\geq \rho_0$, where the last inequality follows directly by the definition of pseudoindex.

\end{proof}

In the next section we apply these results, in particular Corollary \ref{GKM monotone}, to generalized flag varieties. Indeed these are Hamiltonian $T$-spaces whose action
is GKM (see \cite{ghz}), they admit an index increasing component of the moment map (see \cite[Lemma 6.4]{st}) and a symplectic form with respect to which they are positive monotone (see \cite[Proposition 5.24]{gvhs}). In order to obtain the inequalities in \eqref{bound degree H3} it is then sufficient to find the ``pointed box'' of zeros of the generalized Hilbert polynomial, which is the content of Theorem \ref{thm strong}.

\section{Mukai Conjecture for coadjoint orbits and the Kostant game}\label{mukai and kostant}

\subsection{Set up and reformulation of our main claim} Let $G$ be a
compact Lie group. We denote by $G_{\mathbb{C}}$ the
complexification of $G\,.$ Let $P\subset G_{\mathbb{C}}$ be a
parabolic subgroup of $G_{\mathbb{C}}\,.$

Let $\mathfrak{g}$ be the Lie algebra of $G$ and $\mathfrak{g}^*$ be
the dual of $\mathfrak{g}\,.$ We denote by $(\cdot\,,\,\cdot)$ an
Ad-invariant inner product defined on $\mathfrak{g}$ and identify
the Lie algebra $\mathfrak{g}$ and its dual $\mathfrak{g}^*$ via
this inner product.

Let $T\subset G$ be a maximal torus and let $B\subset
G_{\mathbb{C}}$ be a Borel subgroup with $T_{\mathbb{C}}\subset B
\subset P,$ where $T_{\mathbb{C}}$ denotes the complexification of
$T\,.$ Let $R\subset \mathfrak{t}^*$ denote the set of roots and
$R^+$ be the system of positive roots compatible with the choice of
the Borel subgroup $B\subset G_{\mathbb{C}}$ with simple roots
$S\subset R^+\,.$ Let $W:=N_G(T)/T$ be the Weyl group of $G.$ For
every root $\alpha\in R,$ let $s_\alpha\in W$ be the reflection
associated to it. For the parabolic subgroup $P\subset
G_{\mathbb{C}},$ let $W_P:=N_P(T)/T$ be the Weyl group of $P,$
$S_P\subset S$ be the subset of simple roots whose corresponding
reflections are in $W_P$ and $R_P^+$ be the set of positive roots
generated by the simple roots $S_P\,.$

We say that a simple root $\alpha$ is \textit{adjacent} to $P$ if
$\alpha\in S\setminus S_P$ and if in the Dynkin diagram of $G,$ the
simple root $\alpha$ is adjacent to the Dynkin diagram of $S_P\,.$

The set of fundamental weights of $G$ will be denoted by $\{
\varpi_\alpha \mid \alpha \in S\}\,.$ Recall that they are defined
as the dual basis to the basis of coroots
$\check{S}:=\{\alpha^{\vee}= \frac{2 \alpha}{(\alpha, \alpha)}\mid \
\alpha \in S\}$. Just as before, we define $
\check{S}_P:=\{\check{\alpha}\in \check{S} \mid \alpha\in S_P\}$ and
let $\check{R}_P^+$ be the set of positive coroots generated by the
simple coroots $\check{S}_P\,.$

If $\varpi\in \Z\{\varpi_\alpha\mid \alpha\in S\}$ is a weight that
vanishes on all $\beta$ in $S_P,$ it determines a character on $P,$
and so a line bundle $\ll_\varpi=G_{\mathbb{C}}\times^P \C(\varpi)$
on $G_{\mathbb{C}}/P\,.$ We identify the Chern class
$c_1(\ll_\varpi)\in H^2(G_{\mathbb{C}}/P, \Z)$ with the weight
$\varpi$ and we obtain an isomorphism
\begin{align*}
\Z\{\lambda_\alpha\mid \alpha\in \ssp \} &\to H^2(G_{\mathbb{C}}/P,
\Z) \numberthis
\label{secondcohomology} \\
\varpi &\mapsto c_1(\ll_\varpi)\,
\end{align*}
(see for instance \cite{serre}).

We denote by $\hp$ the index map of $G_{\mathbb{C}}/P\,.$ More
precisely, for $\varpi \in \Z\{\varpi_\alpha\mid \alpha\in \ssp
\}\cong \hh^2 (\gp;\Z)$, the value of $\hp(\varpi)$ is the index
$\operatorname{Ind}(\ll_\varpi)$ of the line bundle $\ll_\varpi\,.$
If we write $\varpi=\displaystyle \sum_{\alpha \in \ssp } k_\alpha
\varpi_\alpha,$ Lemma \ref{lemma hilbert polynomial} allows us to
view $\hp(\varpi)$ as a complex polynomial of degree at most the
complex dimension of $G_{\mathbb{C}}/P$ in $ k_\alpha $'s over
$\alpha \in \ssp$. We will call this polynomial the \textit{Hilbert
polynomial of $G_{\mathbb{C}}/P$\,.} The Bott-Borel-Weil Theorem
gives us an explicit formula for the Hilbert polynomial $\hp.$ We
provide this formula and a short explanation of from where it
follows in the next statement.

\begin{theorem}\label{teor formula hp}
For $\varpi=\displaystyle \sum_{\alpha \in \ssp } k_\alpha
\varpi_\alpha \in\Z\{\varpi_\alpha\mid \alpha\in \ssp \} \cong \hh^2
(\gp;\Z),$
$$\
\hp(\varpi)=\operatorname{Ind}(\ll_\varpi)= \prod_{\alpha \in
R_+\setminus R^+_P}\frac{\langle \varpi+\rho,
\alpha^{\vee}\rangle}{\langle\rho, \alpha^{\vee}\rangle},
$$
where $\rho= \displaystyle\frac{1}{2} \sum_{\alpha \in R_+}
\alpha=\sum_{\alpha \in S}\varpi_\alpha$ is half of the sum of
positive roots which is also equal to the sum of all fundamental
weights.
\end{theorem}
\begin{proof}
It follows from the Hirzebruch-Riemann-Roch Thereom that the index
of the bundle $\ll_\varpi$ is the Euler characteristic of
$\ll_\varpi$ in sheaf cohomology, namely the alternating sum
$$
\I (\ll_{\varpi})=\sum_j (-1)^j \dim \hh^j (\gp; \ll_{\varpi})\,.
$$
The Bott-Borel-Weil Theorem implies that for $\varpi$ a dominant
weight, i.e. $ \varpi=\displaystyle \sum_{\alpha \in \ssp} k_\alpha
\varpi_\alpha$ with $k_\alpha\in \Z_{>0},$ higher cohomology
vanishes and thus
$$
\hp(\varpi)=\I (\ll_{\varpi})=\dim \hh^0 (\gp; \ll_{\varpi})
$$
holds for a dominant weight (see for instance \cite{bott},
\cite{kostant}). The Borel-Weil Theorem states that for $\varpi$ a
dominant weight, the action of $G$ on $\hh^0 (\ob; \ll_{\varpi})$ is
the irreducible representation of $G_\C$ with highest weight
$\varpi$ (see e.g. \cite{serre}). The dimension of $\hh^0 (\gp;
\ll_{\varpi})$ for a dominant weight $\varpi$ follows from the Weyl
character formula and equals
$$\prod_{\alpha \in
R_+}\frac{\langle \varpi+\rho, \alpha^{\vee}\rangle}{\langle \rho,
\alpha^{\vee}\rangle}=\prod_{\alpha \in R_+\setminus
R^+_P}\frac{\langle \varpi+\rho, \alpha^{\vee}\rangle}{\langle\rho,
\alpha^{\vee}\rangle}
$$
where $\rho$ is half of the sum of positive roots and we are done.
\end{proof}
The first Chern class $c_1(T (G_\C /P))$ corresponds to
$$
\sum_{\alpha \in R^+ \smallsetminus R_P^+} \alpha=\sum_{\alpha \in
R^+} \alpha-\sum_{\alpha \in  R_P^+} \alpha=\sum_{\alpha \in
S}2\varpi_\alpha -\sum_{\alpha \in R_P^+} \alpha
$$
via the identification in (\ref{secondcohomology}) and can be
written as a linear combination of fundamental weights as
$$
\sum_{\beta \in S} n_\beta \varpi_\beta,
$$
where
$$
n_\beta:=\sum_{\alpha \in R\smallsetminus R_P^+} \langle\alpha,
\beta^{\vee}\rangle=2-\Big\langle\sum_{\alpha \in R_P^+} \alpha,
\beta^{\vee}\Big\rangle\,.
$$

Note for instance that in type A, $n_\beta=0$ if $\beta\in S_P$ and
$n_\beta=2+l$ if $\beta\in S\setminus S_P,$ where $l$ denotes the
sum of the sizes of the connected components of the Dynkin diagram
of $S_P$ adjacent to $\beta\,.$ In particular, if $\beta\in
S\setminus S_P$ is not adjacent to $P,$ then $n_\beta=2\,.$ We
generalise some of these remarks for any type in the following
lemma.

\begin{lemma} If $\beta\in S_P,$ then $n_\beta=0\,.$
If $\beta\in S\setminus S_P$ is a simple root not adjacent to $P,$
then $n_\beta=2\,.$
\end{lemma}
\begin{proof}
If $\beta \in S_P,$ then $\Big\langle\sum_{\alpha \in R_P^+} \alpha,
\beta^{\vee}\Big\rangle=2$ and $n_\beta=0$. Indeed, if we look at
the Dynkin diagram for $G$, and let $\Gamma_P$ be the subgraph
corresponding to $P,$ i.e. the subgraph containing all the simple
roots that are in $P$ and all the edges between them. We call by
$\Gamma_{P_\beta}$ the connected component of $\Gamma_P$ containing
$\beta$ and by $P_\beta$ the parabolic subgroup associated to the
simple roots with Dynkin diagram $\Gamma_{P_\beta}\,.$ The connected
component $\Gamma_P^\beta$ is either a point or a Dynkin diagram of
some classical group. Only the roots $\alpha$ generated by the
simple roots contained in $\Gamma_{P_\beta}$ contribute to
$\langle\sum_{\alpha \in R_P^+} \alpha, \beta^{\vee}\rangle$. For
other roots $\alpha \in R_P^+$ we have that $\langle \alpha,
\beta^{\vee}\rangle=0$. Therefore, for such $\beta$, one has that
$$ \Big\langle\sum_{\alpha \in R_P^+} \alpha,
\beta^{\vee}\Big\rangle=\Big\langle\sum_{\alpha \in R_{P_\beta}^+}
\alpha, \beta^{\vee}\Big\rangle=2\Big\langle\sum_{\alpha \in
S_{P_\beta}} \varpi_\alpha, \beta^{\vee}\Big\rangle=2\,.
$$
If $\beta \notin S_P$ and $\beta$ is not adjacent to any simple root
in $\Gamma_P$, then $\langle\alpha, \beta^{\vee}\rangle=0$ for all
$\alpha \in  R_P^+$ and $n_\beta=2$.
\end{proof}

\begin{theorem}\label{thm strong}
Let $\{\varpi_\alpha\mid \alpha\in S\}$ be the set of fundamental
weights. If we write
$$
c_1(T (G_\C /P))=\sum_{\alpha\in R^+\smallsetminus
R^+_P}\alpha=\sum_{\beta \in S\smallsetminus S_P} n_\beta
\varpi_\beta
$$
where $n_\beta$ is the positive integer $\sum_{\alpha \in
R\smallsetminus R_P^+} \langle\alpha, \beta^{\vee}\rangle,$ then for
each
$$
\varpi=-\sum_{\beta \in S\setminus S_P} \tn_\beta \varpi_\beta
$$
with
$$
0 < \tnb \leq n_\beta\,\,\, \text{but}\,\,\, \sum_{\beta \in
S\setminus S_P} \tnb < \sum_{\beta \in S\setminus S_P} n_\beta,
$$
we have that
$$
\hp(\varpi)=\operatorname{Ind}(\ll_\varpi)=0\,.
$$
\end{theorem}
Note that the last Theorem together with Theorem \ref{theorem
hilbert polynomial} imply the Mukai conjecture for coadjoint orbits
and indeed an stronger inequality.
\begin{corollary}\label{inequality pasquier}
The following inequality
\begin{equation}\label{strong inequality}
\sum_{\beta\in S\smallsetminus S_P }\Big(\sum_{\alpha \in
R\smallsetminus R_P^+} n_\beta-1\Big)=\sum_{\beta\in S\smallsetminus
S_P }\Big(\sum_{\alpha \in R\smallsetminus R_P^+} \langle\alpha,
\beta^{\vee}\rangle -1\Big)\leq \sharp(R^+\smallsetminus R^+_P)
\end{equation}
holds for a parabolic subgroup $P$ of $G_{\mathbb{C}}$ and it
implies the Mukai inequality for $G_{\mathbb{C}}/P$
$$
\sharp(S\smallsetminus S_P)\cdot \big(\gcd_{\beta\in S\smallsetminus
S_P}(n_\beta)-1\big)\leq \sharp(R^+\smallsetminus R^+_P)
$$
as the complex dimension, the second Betti number and the index of
$G_{\mathbb{C}}/P$ equal \\ $\sharp(R^+\smallsetminus R^+_P),
\sharp(S\smallsetminus S_P)$ and $\displaystyle\gcd_{\beta\in
S\smallsetminus S_P}{n_\beta}\,,$ respectively.
\end{corollary}

\begin{remark}
The inequality in \ref{strong inequality} was already proved by B.
Pasquier in \cite[Lemma 4.8]{pasquier}. In Pasquier's proof the
inequality is proven by first reducing its verification to a list of
cases that depend on the type of the Dynkin diagram of the parabolic
subgroup and the position of the simple roots adjacent to the
subgroup and then by checking directly the resulting inequalities
through explicit calculations. The proof that we present in this
paper of the inequality will rely on the combinatorics of the
\textit{Kostant game} and we hope that it will contribute to its
understanding.
\end{remark}

If we view the Hilbert polynomial
$\hp(\varpi)=\hp\big(\displaystyle\sum_{\beta\in S\smallsetminus
S_P} k_\beta \varpi_\beta\big)$ as a polynomial in variables
$k_\beta$ over $\beta\in S\smallsetminus S_P$, Theorem \ref{thm
strong} will follow from the following.

\begin{theorem}\label{roots of hil}
For each $\beta \in S\smallsetminus S_P$ and each $j=1, \ldots,
n_\beta-1$ the Hilbert polynomial
$\hp\Big(\displaystyle\sum_{\beta\in S\smallsetminus S_P} k_\beta
\varpi_\beta\Big)$ vanishes at $k_\beta=-j$.
\end{theorem}
This is the theorem we are going to prove. As part of this theorem,
it is not hard to prove that the statement holds when
$k_\beta=-1\,.$ We present this proof already here.
\begin{lemma}
For each $\beta \in S\smallsetminus S_P$ the Hilbert polynomial
$\hp\Big(\displaystyle\sum_{\beta\in S\smallsetminus S_P} k_\beta
\varpi_\beta\Big)$ vanishes at $k_\beta=-1$.
\end{lemma}
\begin{proof}
By Theorem \ref{teor formula hp}, $\langle \varpi+\rho,
\beta^{\vee}\rangle=k_\beta+1$ is a factor of $\hp$.
\end{proof}
\begin{corollary}\label{notadjacent corollary}
Theorem \ref{roots of hil} holds for a simple root $\beta$ not
adjacent to $P$.
\end{corollary}

\subsection{String of Coroots}

In order to prove Theorem \ref{roots of hil}, we want to show that
$k_\beta+j$ is a linear factor of the Hilbert polynomial
$$\
\hp\Big(\displaystyle\sum_{\beta\in S\smallsetminus S_P} k_\beta
\varpi_\beta\Big)=\operatorname{Ind}(\ll_\varpi)= \prod_{\alpha \in
R_+\setminus R^+_P}\frac{\langle \varpi+\rho,
\alpha^{\vee}\rangle}{\langle\rho, \alpha^{\vee}\rangle}
$$
for $\beta\in S\smallsetminus S_P$ and $j\in \{1, \ldots,
n_\beta-1\}\,.$ That means that for every $\beta\in S\smallsetminus
S_P$ we want to find a string of roots $\alpha_1, \ldots,
\alpha_{n_\beta-1}\in R^+\smallsetminus R^+_P$ such that
$$
\langle \varpi+\rho, \check{\alpha}_j\rangle= \big\langle
\displaystyle\sum_{\beta\in S\smallsetminus S_P} k_\beta
\varpi_\beta +\rho, \check{\alpha}_j\big\rangle=k_\beta+j
$$
for every $j=1, \ldots, n_\beta-1\,.$ Note that it follows from
Corollary \ref{notadjacent corollary} that this task is already
accomplished when $\beta$ is a simple root not adjacent to $P\,.$

The previous analysis motivates the definitions of strings of
coroots for a parabolic subgroup and a simple root adjacent to it
that we define below in this section. Before defining them formally
we set up some notation. Let $P$ be a parabolic subgroup of
$G_{\mathbb{C}}$ and $\beta\in S\smallsetminus S_P$ be a simple root
adjacent to $P\,.$ We denote by $P\cup \{\beta\}$ the parabolic
subgroup corresponding to $S_P\cup \{\beta\}\,.$ Let $\Gamma$ be the
Dynkin diagram of $G\,.$ Let $\Gamma_P$ and $\Gamma_P \cup \beta$ be
the subgraphs of $\Gamma$ corresponding to the sets of simple roots
$S_P$ and $S_P\cup \{\beta\},$ respectively.

\begin{definition}
A {\bf string of coroots} (for $P$ and $\beta$) is a string of
positive coroots of the form
$$\check{\beta}, \check{\beta}+ \gamma_1, \ldots, \check{\beta}+\gamma_l \in \check{R}_{P\cup \beta}^+,$$
such that every $\gamma_j$ can be written as a positive integer
linear combination of coroots in $\check{R}_P^+$ and
$$
\operatorname{ht}(\check{\beta} + \gamma_j)<
\operatorname{ht}(\check{\beta} + \gamma_{j+1})
$$
for every $1\leq j< l\,,$ where $\operatorname{ht}$ stands for the
height function defined on the set of positive coroots. The integer
$\operatorname{ht}(\check{\beta}+\gamma_l)$ will be called the
\textbf{length of the string}. A string of coroots is called {\bf
maximal } if its length is exactly equal to
$$
n_\beta-1=\sum_{\alpha\in R^+\smallsetminus R^+_P}\langle\alpha,
\check{\beta}\rangle-1=1-\sum_{\alpha \in R^+_P} \langle\alpha,
\check{\beta}\rangle\,.
$$
A string of coroots is called \textbf{good} if
$$
\operatorname{ht}(\check{\beta} + \gamma_{j})=j+1
$$
for every $1\leq j\leq l\,.$ A string of coroots that is not good is
called a \textbf{string with a jump}. A \textbf{gap of a string with
a jump} is an integer $j$ such that $\operatorname{ht}(\check{\beta}
+ \gamma_j)+1< \operatorname{ht}(\check{\beta} + \gamma_{j+1})\,.$

\end{definition}

\begin{proposition}\label{coroots from good strings}
If $\check{\beta}, \check{\beta}+\gamma_1, \ldots,
\check{\beta}+\gamma_l$ is a good string of coroots for $P$ and
$\beta$, then the Hilbert polynomial
$\hp(\varpi)=\hp\Big(\displaystyle\sum_{\beta\in S\smallsetminus
S_P} k_\beta \varpi_\beta\Big)$ vanishes at $k_\beta=-1, \ldots,
-(l+1)$.
\end{proposition}
\begin{proof}
For each $j=1,\ldots,l$ there is a factor $\langle
\varpi+\rho,\check{\beta}+ \gamma_j\rangle$ in $\hp$ which, by the
conditions satisfied by good strings, is equal to
$$\langle \varpi+\rho,\check{\beta}+ \gamma_j\rangle=k_\beta + (j+1)$$
thus $\hp$ vanishes at $k_\beta=-(j+1)$.
\end{proof}

In order to prove Theorem \ref{roots of hil}, our aim now is to find
for every simple root adjacent to a parabolic subgroup a maximal
good string of coroots for the parabolic subgroup and the simple
root. Maximal good strings of coroots can be obtained via Kostant
games which we define in the following section.

\subsection{The Kostant game}\label{sc:kostant game}

Let $\Gamma$ be the Dynkin diagram of $G$ and $S=\{\alpha_1, \ldots,
\alpha_n\}$ be the set of simple roots.

A configuration is a vector $c=\sum_{i=1}^n c_i \alpha_i$ where
$c_i\in \Z_{\geq 0}\,.$ The \textbf{height $\operatorname{ht}(c)$}
of $c$ is the sum $\sum_{i=1}^nc_i\,.$ We can think of a
configuration as placing $c_i$ chips on the $i$-vertex of the Dynkin
diagram. The height of the configuration is the total number of
chips placed on the diagram.

For $i\in V,$ let $N(i)$ denote the neighbors of $i\,.$ For $j\in
N(i)$, we denote by $n_{i,j}$ the number of arrows coming to the
$i$-vertex from the $j$-vertex. In our convention, when the simple
roots $\alpha_i$ and $\alpha_j$ have the same length, then $n_{i,
j}=1=n_{j, i}.$ Also in our convention, for instance in the figure
below  $n_{i, j}=2$ and $n_{j, i}=1$\,.
\begin{figure}[!ht]
\centering
\begin{tikzpicture}[scale=1.4,transform shape]
  \draw (1,-0.06) -- (2,-0.06);
  \draw (1,0.06) -- (2,0.06);
 \node[draw,fill=white,shape=circle,scale=0.7] (b) at (1,0) {};
  \node[draw,fill=white,shape=circle,scale=0.7] (c) at (2,0) {};
 \node [scale=0.7,below] at (1,-0.2) {$j$};
 \node [scale=0.7,below] at (2,-0.2) {$i$};
 \draw (1.42,0.16) -- (1.58,0) -- (1.42,-0.16);
\end{tikzpicture}
\end{figure}
Note that $n_{i,j}=-\langle \alpha_j, \check{\alpha}_i\rangle\,.$

Let $c=\sum_{i=1}^nc_i\alpha_i$ be a configuration. We say that the
$i$-th entry of the configuration is:

\begin{itemize}
\item Happy if $c_i=\dfrac{1}{2}\sum_{j\in N(i)}n_{i, j}c_j\,.$

\item Unhappy if $c_i<\dfrac{1}{2}\sum_{j\in N(i)}n_{i, j}c_j\,.$

\item Excited if $c_i>\dfrac{1}{2}\sum_{j\in N(i)}n_{i, j}c_j\,.$
\end{itemize}

We play the Kostant game by first placing a chip at a vertex of the
Dynkin diagram. The goal of the Kostant game is to make every vertex
of the Dynkin diagram either happy or excited. The game is played as
follows. Given a configuration $c,$ we pick any unhappy vertex $i,$
and replace the number of chips $c_i$ at $i$ by
$$
c_i\mapsto -c_i+\sum_{j\in N(i)}n_{i, j}c_j\,.
$$
Equivalently, we replace $c$ by
\begin{align*}
s_{\alpha_i}(c):&=c-\langle c, \check{\alpha}_i\rangle\alpha_i=
\sum_{j=1}^n c_j \alpha_j - \langle \sum_{j=1}^n c_j \alpha_j,
\check{\alpha}_i\rangle\alpha_i\\
&=\sum_{j\ne i}c_j\alpha_j+\Big(-c_i+\sum_{j\in
N(i)}n_{i,j}c_j\Big)\alpha_i\,.
\end{align*}
Note that if the vertex $i$ is unhappy, then
$\operatorname{ht}(s_{\alpha_i}(c))>\operatorname{ht}(c)\,.$ It
follows from this remark that we play the Kostant game by
consecutively replacing one positive root by another of greatest
height. Indeed, the set of all the possible configurations of the
Kostant game played by initially placing a chip on a vertex of the
Dynkin diagram corresponds to the set of positive roots $R^+$. When
we start playing the game by placing one chip at one of the
vertices, the game will reach the same terminating configuration
regardless of the sequence of moves. If the Dynkin diagram is simply
laced, the game always leads to a unique terminating configuration
regardless of the vertex where we place the first chip. This unique
terminating configuration corresponds to the highest positive root.
If the Dynkin diagram is not simply-laced, any sequence of moves of
the Kostant game will lead to two possible terminating
configurations, and one of them corresponds to the highest positive
root. For these and other facts concerning the combinatorics of the
Kostant game we suggest the reader to check the class notes written
by E. Chen of the course \textit{``Topics on Combinatorics''} taught
by A. Postnikov at 2017 in MIT \cite{postnikov} and B. Elek's notes
on Reflection Groups \cite[Section 5.4]{elek}.

\begin{example}
In this example we play the Kostant game on the Dynkin diagram $A_4$
starting from its simple roots until we reach the configuration of
the highest positive root.

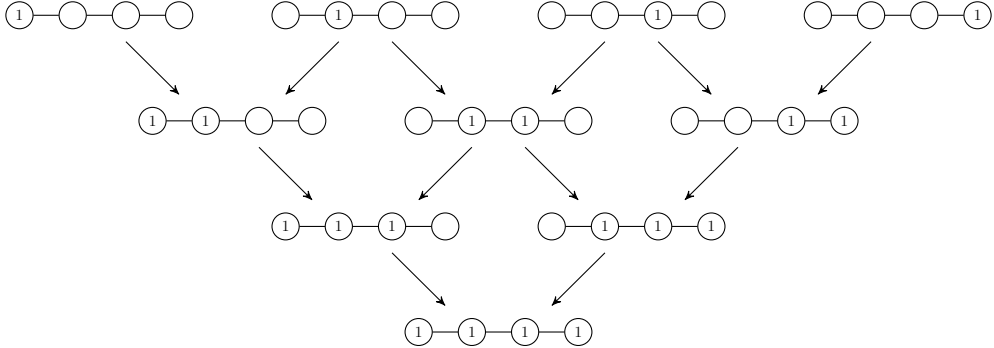
\begin{figure}[!ht]
\centering
\begin{tikzpicture}[scale=0.7,transform shape]
  \node[draw,shape=circle,scale=0.7] (a) at (-7,0) {$1$};
  \node[draw,fill=white,shape=circle,scale=0.7] (b) at (-6,0) {\phantom{$0$}};
  \node[draw,fill=white,shape=circle,scale=0.7] (c) at (-5,0) {\phantom{$0$}};
  \node[draw,fill=white,shape=circle,scale=0.7] (d) at (-4,0) {\phantom{$0$}};
  \draw(a) -- (b) -- (c) -- (d);
  \draw [->] (-5,-0.5) -- (-4,-1.5);

  \node[draw,shape=circle,scale=0.7] (a) at (-2,0) {\phantom{$0$}};
  \node[draw,fill=white,shape=circle,scale=0.7] (b) at (-1,0) {$1$};
  \node[draw,fill=white,shape=circle,scale=0.7] (c) at (0,0) {\phantom{$0$}};
  \node[draw,fill=white,shape=circle,scale=0.7] (d) at (1,0) {\phantom{$0$}};
   \draw(a) -- (b)--(c)-- (d);
   \draw [->] (-1,-0.5) -- (-2,-1.5);
  \draw [->] (0,-0.5) -- (1,-1.5);

  \node[draw,shape=circle,scale=0.7] (a) at (3,0) {\phantom{$0$}};
  \node[draw,fill=white,shape=circle,scale=0.7] (b) at (4,0) {\phantom{$0$}};
  \node[draw,fill=white,shape=circle,scale=0.7] (c) at (5,0) {$1$};
 \node[draw,fill=white,shape=circle,scale=0.7] (d) at (6,0) {\phantom{$0$}};
  \draw(a) -- (b)--(c)--(d);
  \draw [->] (4,-0.5) -- (3,-1.5);
  \draw [->] (5,-0.5) -- (6,-1.5);

  \node[draw,shape=circle,scale=0.7] (a) at (8,0) {\phantom{$0$}};
  \node[draw,fill=white,shape=circle,scale=0.7] (b) at (9,0) {\phantom{$0$}};
  \node[draw,fill=white,shape=circle,scale=0.7] (c) at (10,0) {\phantom{$0$}};
  \node[draw,fill=white,shape=circle,scale=0.7] (d) at (11,0) {$1$};
  \draw(a) -- (b)--(c)--(d);
  \draw [->] (9,-0.5) -- (8,-1.5);


  \node[draw,shape=circle,scale=0.7] (a) at (-4.5,-2) {$1$};
  \node[draw,fill=white,shape=circle,scale=0.7] (b) at (-3.5,-2) {$1$};
  \node[draw,fill=white,shape=circle,scale=0.7] (c) at (-2.5,-2) {\phantom{$0$}};
  \node[draw,fill=white,shape=circle,scale=0.7] (d) at (-1.5,-2) {\phantom{$0$}};
  \draw(a) -- (b)--(c)--(d);
  \draw [->] (-2.5,-2.5) -- (-1.5,-3.5);

  \node[draw,shape=circle,scale=0.7] (a) at (0.5,-2) {\phantom{$0$}};
  \node[draw,fill=white,shape=circle,scale=0.7] (b) at (1.5,-2) {$1$};
  \node[draw,fill=white,shape=circle,scale=0.7] (c) at (2.5,-2) {$1$};
  \node[draw,fill=white,shape=circle,scale=0.7] (d) at (3.5,-2) {\phantom{$0$}};
  \draw(a) -- (b)--(c)--(d);
  \draw [->] (1.5, -2.5) -- (0.5, -3.5);
  \draw [->] (2.5, -2.5) -- (3.5, -3.5);

  \node[draw,shape=circle,scale=0.7] (a) at (5.5,-2) {\phantom{$0$}};
  \node[draw,fill=white,shape=circle,scale=0.7] (b) at (6.5,-2) {\phantom{$0$}};
  \node[draw,fill=white,shape=circle,scale=0.7] (c) at (7.5,-2) {$1$};
  \node[draw,fill=white,shape=circle,scale=0.7] (d) at (8.5,-2) {$1$};
  \draw(a) -- (b)--(c)--(d);
  \draw [->] (6.5, -2.5) -- (5.5, -3.5);


  \node[draw,shape=circle,scale=0.7] (a) at (-2,-4) {$1$};
  \node[draw,fill=white,shape=circle,scale=0.7] (b) at (-1,-4) {$1$};
  \node[draw,fill=white,shape=circle,scale=0.7] (c) at (-0,-4) {$1$};
  \node[draw,fill=white,shape=circle,scale=0.7] (d) at (1,-4) {\phantom{$0$}};
  \draw(a) -- (b)--(c)--(d);
  \draw [->] (0, -4.5) -- (1, -5.5);

  \node[draw,shape=circle,scale=0.7] (a) at (3,-4) {\phantom{$0$}};
  \node[draw,fill=white,shape=circle,scale=0.7] (b) at (4,-4) {$1$};
  \node[draw,fill=white,shape=circle,scale=0.7] (c) at (5,-4) {$1$};
  \node[draw,fill=white,shape=circle,scale=0.7] (d) at (6,-4) {$1$};
  \draw(a) -- (b)--(c)--(d);
   \draw [->] (4, -4.5) -- (3, -5.5);


  \node[draw,shape=circle,scale=0.7] (a) at (0.5,-6) {$1$};
  \node[draw,fill=white,shape=circle,scale=0.7] (b) at (1.5,-6) {$1$};
  \node[draw,fill=white,shape=circle,scale=0.7] (c) at (2.5,-6) {$1$};
  \node[draw,fill=white,shape=circle,scale=0.7] (d) at (3.5,-6) {$1$};
  \draw(a) -- (b)--(c)--(d);

\end{tikzpicture}
\end{figure}

\end{example}

\subsection{The modified Kostant game}\label{sc:modified kostant game}

We modify the Kostant game on Dynkin diagrams of compact Lie groups
in order to construct maximal good strings. We modify the graph of
$\Gamma$ by adding an extra vertex adjacent only to a vertex $j$
of the graph. We denote the new vertex by $\tilde{j}\,.$ We also
draw $k$-arrows pointing from the new vertex $\tilde{j}$ to $j\,.$
We denote the new graph by $\Gamma_j^k\,.$ When $k=1,$ we just
denote it by $\Gamma_j\,.$ We play a modified version of the Kostant
game on $\Gamma_j^k$ by placing one chip at the vertex $\tilde{j},$
with the same rules as before but with a new rule in which the
vertex $\tilde{j}$ is \textit{always} happy. We call the resulting
game on $\Gamma_j^k$ the \textbf{modified Kostant game} at the
vertex $j$ with $k$-arrows. When $k=1,$ we just call the resulting
game the modified Kostant game at the vertex $j\,.$ We define
configurations and their heights as we did for the standard Kostant
game.

The proofs of the two theorems below will be given in the later
section.

\begin{theorem}\label{Main theorem1}
The modified Kostant game at a vertex of a Dynkin diagram of a
compact Lie groups leads to a unique terminating configuration.
\end{theorem}

\begin{remark}

Note that given a Dynkin diagram, if we denote by $h+1$ the height
of the final configuration of the modified Kostant game on the
Dynkin diagram at a fixed vertex, then the height of the terminating
configuration of the modified Kostant game at the same vertex but
with $k$-arrows equals $kh+1\,.$
\end{remark}

\begin{theorem}\label{Main theorem2}
Given a Dynkin diagram $\Gamma$ of a compact Lie group $G,$ if we
denote by $h_j+1$ the height of the unique terminating configuration
of the modified Kostant game on the Dynkin diagram of simple coroots
$\check{\Gamma}$ at the vertex $j,$ then
$$
\sum_{\alpha\in R^+}\alpha=\sum_{\alpha_j\in S}h_j \alpha_j\,.
$$
\end{theorem}

\begin{example}
We illustrate Theorem \ref{Main theorem2} with one example. Let
$\Gamma=A_4.$ We enumerate the vertices of $A_4$ from the left to
the right in increasing order. We denote by $h_i+1, i=1, \ldots, 4,$
the height of the unique terminating configuration of the modified
Kostant game on $\Gamma=\check{\Gamma}$ at the vertex $i\,.$ By
symmetry, $h_1=h_4$ and $h_2=h_3\,.$ We play the modified Kostant
game at the first and second vertex and illustrate all its possible
configurations in the figure below.

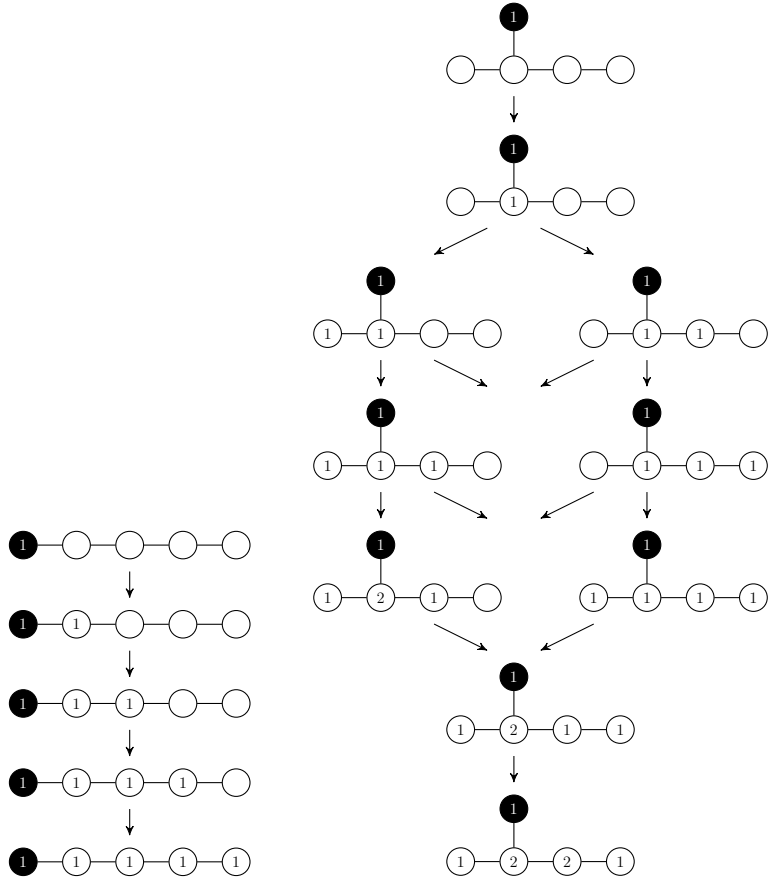
\begin{center}
\begin{tikzpicture}[scale=0.7]
  \node[draw,fill=black,shape=circle,scale=0.5] (a) at (0,0) {$\color{white}{1}$};
  \node[draw,fill=white,shape=circle,scale=0.5] (b) at (1,0) {\phantom{$0$}};
  \node[draw,fill=white,shape=circle,scale=0.5] (c) at (2,0) {\phantom{$0$}};
  \node[draw,fill=white,shape=circle,scale=0.5] (d) at (3,0) {\phantom{$0$}};
  \node[draw,fill=white,shape=circle,scale=0.5] (e) at (4,0) {\phantom{$0$}};
  \draw (a)--(b)--(c)--(d)--(e);
  \draw [->] (2,-0.5) -- (2,-1);
  \node[draw,fill=black,shape=circle,scale=0.5] (a) at (0,-1.5) {$\color{white}{1}$};
  \node[draw,fill=white,shape=circle,scale=0.5] (b) at (1,-1.5) {$1$};
  \node[draw,fill=white,shape=circle,scale=0.5] (c) at (2,-1.5) {\phantom{$0$}};
  \node[draw,fill=white,shape=circle,scale=0.5] (d) at (3,-1.5) {\phantom{$0$}};
  \node[draw,fill=white,shape=circle,scale=0.5] (e) at (4,-1.5) {\phantom{$0$}};
  \draw (a)--(b)--(c)--(d)--(e);
  \draw [->] (2,-2) -- (2,-2.5);
  \node[draw,fill=black,shape=circle,scale=0.5] (a) at (0,-3)
  {$\color{white}{1}$};
  \node[draw,fill=white,shape=circle,scale=0.5] (b) at (1,-3) {$1$};
  \node[draw,fill=white,shape=circle,scale=0.5] (c) at (2,-3) {$1$};
  \node[draw,fill=white,shape=circle,scale=0.5] (d) at (3,-3) {\phantom{$0$}};
  \node[draw,fill=white,shape=circle,scale=0.5] (e) at (4,-3) {\phantom{$0$}};
  \draw (a)--(b)--(c)--(d)--(e);
  \draw [->] (2,-3.5) -- (2,-4);
  \node[draw,fill=black,shape=circle,scale=0.5] (a) at (0,-4.5)
  {$\color{white}{1}$};
  \node[draw,fill=white,shape=circle,scale=0.5] (b) at (1,-4.5) {$1$};
  \node[draw,fill=white,shape=circle,scale=0.5] (c) at (2,-4.5) {$1$};
  \node[draw,fill=white,shape=circle,scale=0.5] (d) at (3,-4.5) {$1$};
  \node[draw,fill=white,shape=circle,scale=0.5] (e) at (4,-4.5) {\phantom{$0$}};
  \draw (a)--(b)--(c)--(d)--(e);
  \draw [->] (2,-5) -- (2,-5.5);
  \node[draw,fill=black,shape=circle,scale=0.5] (a) at (0,-6)
  {$\color{white}{1}$};
  \node[draw,fill=white,shape=circle,scale=0.5] (b) at (1,-6) {$1$};
  \node[draw,fill=white,shape=circle,scale=0.5] (c) at (2,-6) {$1$};
  \node[draw,fill=white,shape=circle,scale=0.5] (d) at (3,-6) {$1$};
  \node[draw,fill=white,shape=circle,scale=0.5] (e) at (4,-6) {$1$};
  \draw (a)--(b)--(c)--(d)--(e);
\end{tikzpicture}
\qquad 
\begin{tikzpicture}[scale=0.7]
  \node[draw,shape=circle,scale=0.5] (a) at (-0.5,0) {\phantom{$0$}};
  \node[draw,fill=white,shape=circle,scale=0.5] (b) at (0.5,0) {\phantom{$0$}};
  \node[draw,fill=white,shape=circle,scale=0.5] (c) at (1.5,0) {\phantom{$0$}};
  \node[draw,fill=black,shape=circle,scale=0.5] (d) at (-0.5,1) {$\color{white}{1}$};
  \node[draw,fill=white,shape=circle,scale=0.5] (e) at (-1.5,0) {\phantom{$0$}};
  \draw (e) -- (a) -- (b)--(c);
  \draw (a) -- (d);
  \draw [->] (-0.5,-0.5) -- (-0.5,-1);
\node[draw,shape=circle,scale=0.5] (a) at (-0.5,-2.5) {$1$};
  \node[draw,fill=white,shape=circle,scale=0.5] (b) at (0.5,-2.5) {\phantom{$0$}};
  \node[draw,fill=white,shape=circle,scale=0.5] (c) at (1.5,-2.5) {\phantom{$0$}};
  \node[draw,fill=black,shape=circle,scale=0.5] (d) at (-0.5,-1.5) {$\color{white}{1}$};
  \node[draw,fill=white,shape=circle,scale=0.5] (e) at (-1.5,-2.5) {\phantom{$0$}};
  \draw (e) -- (a) -- (b)--(c);
  \draw (a) -- (d);
  \draw [->] (-1,-3) -- (-2,-3.5);
  \draw [->] (0,-3) -- (1,-3.5);
\node[draw,shape=circle,scale=0.5] (a) at (-3,-5) {$1$};
  \node[draw,fill=white,shape=circle,scale=0.5] (b) at (-2,-5) {\phantom{$0$}};
  \node[draw,fill=white,shape=circle,scale=0.5] (c) at (-1,-5) {\phantom{$0$}};
  \node[draw,fill=black,shape=circle,scale=0.5] (d) at (-3,-4) {$\color{white}{1}$};
  \node[draw,fill=white,shape=circle,scale=0.5] (e) at (-4,-5) {$1$};
  \draw (e) -- (a) -- (b)--(c);
  \draw (a) -- (d);
\node[draw,shape=circle,scale=0.5] (a) at (2,-5) {$1$};
  \node[draw,fill=white,shape=circle,scale=0.5] (b) at (3,-5) {$1$};
  \node[draw,fill=white,shape=circle,scale=0.5] (c) at (4,-5) {\phantom{$0$}};
  \node[draw,fill=black,shape=circle,scale=0.5] (d) at (2,-4) {$\color{white}{1}$};
  \node[draw,fill=white,shape=circle,scale=0.5] (e) at (1,-5) {\phantom{$0$}};
  \draw (e) -- (a) -- (b)--(c);
  \draw (a) -- (d);
  \draw [->] (-2,-5.5) -- (-1,-6);
  \draw [->] (1,-5.5) -- (0,-6);
  \draw [->] (-3,-5.5) -- (-3,-6);
  \draw [->] (2,-5.5) -- (2,-6);
  \node[draw,shape=circle,scale=0.5] (a) at (-3,-7.5) {$1$};
  \node[draw,fill=white,shape=circle,scale=0.5] (b) at (-2,-7.5) {$1$};
  \node[draw,fill=white,shape=circle,scale=0.5] (c) at (-1,-7.5) {\phantom{$0$}};
  \node[draw,fill=black,shape=circle,scale=0.5] (d) at (-3,-6.5) {$\color{white}{1}$};
  \node[draw,fill=white,shape=circle,scale=0.5] (e) at (-4,-7.5){$1$};
  \draw (e) -- (a) -- (b)--(c);
  \draw (a) -- (d);
  \node[draw,shape=circle,scale=0.5] (a) at (2,-7.5) {$1$};
  \node[draw,fill=white,shape=circle,scale=0.5] (b) at (3,-7.5) {$1$};
  \node[draw,fill=white,shape=circle,scale=0.5] (c) at (4,-7.5) {$1$};
  \node[draw,fill=black,shape=circle,scale=0.5] (d) at (2,-6.5) {$\color{white}{1}$};
  \node[draw,fill=white,shape=circle,scale=0.5] (e) at (1,-7.5){\phantom{$0$}};
  \draw (e) -- (a) -- (b)--(c);
  \draw (a) -- (d);
  \draw [->] (-2,-8) -- (-1,-8.5);
  \draw [->] (1,-8) -- (0,-8.5);
  \draw [->] (-3,-8) -- (-3,-8.5);
  \draw [->] (2,-8) -- (2,-8.5);
  \node[draw,shape=circle,scale=0.5] (a) at (-3,-10) {$2$};
  \node[draw,fill=white,shape=circle,scale=0.5] (b) at (-2,-10) {$1$};
  \node[draw,fill=white,shape=circle,scale=0.5] (c) at (-1,-10) {\phantom{$0$}};
  \node[draw,fill=black,shape=circle,scale=0.5] (d) at (-3,-9) {$\color{white}{1}$};
  \node[draw,fill=white,shape=circle,scale=0.5] (e) at (-4,-10){$1$};
  \draw (e) -- (a) -- (b)--(c);
  \draw (a) -- (d);
  \node[draw,shape=circle,scale=0.5] (a) at (2,-10) {$1$};
  \node[draw,fill=white,shape=circle,scale=0.5] (b) at (3,-10) {$1$};
  \node[draw,fill=white,shape=circle,scale=0.5] (c) at (4,-10) {$1$};
  \node[draw,fill=black,shape=circle,scale=0.5] (d) at (2,-9) {$\color{white}{1}$};
  \node[draw,fill=white,shape=circle,scale=0.5] (e) at (1,-10){$1$};
  \draw (e) -- (a) -- (b)--(c);
  \draw (a) -- (d);
   \draw [->] (-2,-10.5) -- (-1,-11);
  \draw [->] (1,-10.5) -- (0,-11);
  \node[draw,shape=circle,scale=0.5] (a) at (-0.5,-12.5) {$2$};
  \node[draw,fill=white,shape=circle,scale=0.5] (b) at (0.5,-12.5) {$1$};
  \node[draw,fill=white,shape=circle,scale=0.5] (c) at (1.5,-12.5) {$1$};
  \node[draw,fill=black,shape=circle,scale=0.5] (d) at (-0.5,-11.5) {$\color{white}{1}$};
  \node[draw,fill=white,shape=circle,scale=0.5] (e) at (-1.5,-12.5) {$1$};
  \draw (e) -- (a) -- (b)--(c);
  \draw (a) -- (d);
  \draw [->] (-0.5,-13) -- (-0.5,-13.5);
  \node[draw,shape=circle,scale=0.5] (a) at (-0.5,-15) {$2$};
  \node[draw,fill=white,shape=circle,scale=0.5] (b) at (0.5,-15) {$2$};
  \node[draw,fill=white,shape=circle,scale=0.5] (c) at (1.5,-15) {$1$};
  \node[draw,fill=black,shape=circle,scale=0.5] (d) at (-0.5,-14) {$\color{white}{1}$};
  \node[draw,fill=white,shape=circle,scale=0.5] (e) at (-1.5,-15) {$1$};
  \draw (e) -- (a) -- (b)--(c);
  \draw (a) -- (d);
\end{tikzpicture}

\end{center}

From the figures above we conclude that $h_1=h_4=4, h_2=h_3=6$ and
Theorem \ref{Main theorem2} states that the sum of the positive
roots of $A_4$ equals
$$ \sum_{\alpha\in
R^+}\alpha=4\alpha_1+6\alpha_2+6\alpha_3+4\alpha_4\,.
$$

\end{example}

We would like to prove the existence of maximal good strings for a
parabolic group and a root adjacent to it from Theorem \ref{Main
theorem1} and Theorem \ref{Main theorem2}. Before doing so we review
the following two lemmas on roots whose proofs can be found for
instance in \cite[Section 9.4]{Humphreys1}.
\begin{lemma}
Let $\alpha, \beta$ be non-proportional roots. If $(\alpha,
\beta)>0$, then $\alpha-\beta$ is a root. If $(\alpha, \beta)<0,$
then $\alpha+\beta\,$ is a root.
\end{lemma}

\begin{lemma}\label{lemma9.4H}
Let $\alpha$ and $\beta$ be non-proportional roots. Let $r, q\in
\Z_{\geq 0}$ be the largest integers for which $\beta-r\alpha\in R,
\beta+q\alpha\in R.$ Then $\beta+i\alpha\in R$ for all $-r\leq i\leq
q\,.$ We call the set of roots $\beta+i\alpha\, (i\in \Z)$ the
\textbf{$\alpha$-string through $\beta\,.$}
\end{lemma}

\begin{theorem}\label{maximal good strings}
Let $\Gamma$ be a Dynkin diagram and $\Gamma_P$ be a connected
subgraph. Let $\beta$ be a simple root adjacent to $\Gamma_P$ and
let $\alpha_j$ be the root in $\Gamma_P$ adjacent to $\beta\,.$ Let
us assume that there are $k$-arrows pointing from $\alpha_j$ to
$\beta\,.$

When we play the modified Kostant game on the Dynkin diagram of
coroots of $\Gamma_P$ at the $j$-vertex with $k$-arrows until it
reaches its terminating configuration, we obtain a maximal string of
coroots
$$
\check{\beta}, \check{\beta}+\gamma_1, \ldots,
\check{\beta}+\gamma_l\,.
$$
for $P$ and $\beta\,.$ This string of coroots can always be
completed into a good string of coroots for $P$ and $\beta\,.$
\end{theorem}
\begin{proof}
The string of coroots
$$
\check{\beta}, \check{\beta}+\gamma_1, \ldots,
\check{\beta}+\gamma_l\,.
$$
is indeed a string of coroots for $P$ and $\beta$ because the string
is obtained by playing the standard Kostant game on the coroots of
$\Gamma_P\cup \beta$ by leaving always one chip at the
$\beta$-vertex.

We write
$$
\sum_{\alpha\in R_P^+}\alpha=h_j\alpha_j+\ldots\,.
$$
The coroot $\check{\beta}$ is orthogonal to every root in $S_P$
except $\alpha_j,$ thus
$$
\bigl\langle\sum_{\alpha\in R^+_P}\alpha,
\check{\beta}\bigr\rangle=\langle h_j\alpha_j+\ldots,
\check{\beta}\rangle=h_j\langle \alpha_j, \check{\beta}\rangle=-k
h_j.
$$
The string of coroots is maximal because Theorem \ref{Main theorem2}
implies that
$$
\operatorname{ht}(\check{\beta}+\gamma_l)=k
h_j+1=1-\bigl\langle\sum_{\alpha\in \Gamma^+_P}\alpha,
\check{\beta}\bigr\rangle\,.
$$
Let us suppose that
$$
\check{\beta}, \check{\beta}+\gamma_1, \ldots,
\check{\beta}+\gamma_l\,
$$
is a string with a jump. That means that there exist $r$ such that
$$
\operatorname{ht}(\check{\beta}+\gamma_r)+1<\operatorname{ht}(\check{\beta}+\gamma_{r+1})\,.
$$
As we have obtained the string by playing the Kostant game, there
exist a simple root $\alpha\in S_P$ and a positive integer $n$ such
that
$$
\check{\beta}+\gamma_{r+1}=\check{\beta}+\gamma_r+n\check{\alpha}\,.
$$
Lemma \ref{lemma9.4H} implies that
$$
\check{\beta}+\gamma_r+m\check{\alpha}
$$
is a coroot for every $0\leq m\leq n\,$ (a $\check{\alpha}$-string
through $\check{\beta}+\gamma_r$ is unbroken). We can always fill
all the gaps of the maximal string in this way to get a good string
of coroots for $P$ and $\beta\,,$ and we are done.
\end{proof}

\begin{proof}[Proof of Theorem \ref{roots of hil}]
Let $\Gamma$ be the Dynkin diagram of $G$ and $\Gamma_P$ be the
subgraph corresponding to $P\,$ (not necessarily connected). Assume
that $\beta$ is a simple root adjacent to $P$. We show that there
exists a maximal good string of coroots for $P$ and $\beta$ and the
Theorem will follow from Proposition \ref{coroots from good strings}\,.

We consider the connected components
$\Gamma_{P_j}=\Gamma_{P_j}(\beta)$ of $\Gamma_P$ which are adjacent
to $\beta\,.$ Note that there are at most three such components. Let
$$
\check{\beta},\check{\beta}+\gamma_1^j,\ldots,\check{\beta}+\gamma_{l_j}^j\,.
$$
be a maximal string of coroots for $P_j$ and $\beta\,.$ We know from
the previous Theorem that they exist. The mechanics of the standard
Kostant game and Lemma \ref{lemma9.4H} imply that we can glue the
strings to obtain a good string of coroots, for instance the
following is a good string of coroots for $P$ and $\beta$
\begin{align*}
\check{\beta},\check{\beta}+\gamma_1^1,\ldots,\check{\beta}+\gamma_{l_1}^1,
\check{\beta}+\gamma_{l_1}^1+\gamma_1^{2},\ldots,
\check{\beta}+\gamma_{l_1}^1+\gamma_{l_2}^{2},\ldots,\check{\beta}+\sum_{j}\gamma_{l_j}^j\,.
\end{align*}
Now we verify that it is maximal:
\begin{align*}
n_\beta-1&=1-\sum_{\alpha\in R_P^+}\langle \alpha,
\check{\beta}\rangle=1-\sum_j\sum_{\alpha\in R_{P_j}^+}\langle
\alpha,
\check{\beta}\rangle\\&=1+\sum_j\Bigl(\operatorname{ht}(\check{\beta}+\gamma_{l_j}^j)-1\Bigr)
=\operatorname{ht}\Bigl(\check{\beta}+\sum_j\gamma_{l_j}^j\Bigr)\,.
\end{align*}
The Theorem now follows from Proposition \ref{coroots from good
strings}\,.

\end{proof}

\subsection{Proofs of Theorem \ref{Main theorem1} and Theorem
\ref{Main theorem2}}

In this section we give the proofs of Theorem \ref{Main theorem1}
and Theorem \ref{Main theorem2}. It turns out that the
configurations of a modified Kostant game correspond to minimal
length representatives of the quotient of the Weyl group of a
compact Lie group with a parabolic subgroup. We show Theorem
\ref{Main theorem1} and Theorem \ref{Main theorem2} using this fact
together with the fact that the set of minimal length
representatives of the quotient contains a unique maximal length
representative whose set of inversions is the set of positive roots
not belonging to the set of positive roots spanned by the simple
roots in the parabolic. Before giving formal explanations and proofs
of our claims, we recall some notation.

Let $G$ be a compact Lie group and $T$ be a maximal torus. Let $R$
be the root system defined by $G$ and $T$ with a choice of positive
roots $R^+$ and simple roots $S=\{\alpha_i\}_{i=1}^n.$ We denote by
$\Gamma$ the corresponding Dynkin diagram. We identify the Lie
algebra $\mathfrak{t}$ of $T$ with its dual $\mathfrak{t}^*$ via an
invariant inner product $( \cdot,\cdot )$ on the Lie algebra
$\mathfrak{g}$ of $G\,.$ For every $\alpha\in R,$ we denote by
$s_\alpha$ the reflection defined by $\alpha:$
\begin{align*}
s_\alpha:\mathfrak{t} &\to \mathfrak{t} \\
x &\mapsto x-\langle x, \alpha\rangle\, \check{\alpha}\\
&=x-2\dfrac{(x, \alpha)}{(\alpha, \alpha)}\alpha\,.
\end{align*}
We denote by $W$ the Weyl group generated by the set of simple
reflections $\{s_\alpha\mid \alpha\in S \}\,.$ When
$\alpha=\alpha_i$ is a simple root, we denote the corresponding
reflection just by $s_{\alpha_i}=s_i\,.$

We define the \textbf{length} $l(w)$ of an element $w\in W$ as the
smallest positive integer such that $w$ can be written as a product
of simple reflections
$$
w=s_{i_t}\cdot\ldots\cdot s_{i_k}\cdot\ldots\cdot s_{i_1}
$$
and call such an expression \textbf{reduced}.

The length of an element in the Weyl group can be characterized in
another way. Let
$$
I(w)=\{\alpha\in  R^+ \mid w(\alpha)\in -R^+\}
$$
be the \textbf{inversion set of $w$}. Then
$l(w)=\operatorname{Car}(I(w))\,.$ A reduced expression
$s_{i_t}\ldots\cdot s_{i_k}\cdot\ldots s_{i_1}$ of $w$ allows us to
enumerate all the positive roots in $I(w)$ as follows: define
$$
\tilde{\alpha}_k=s_{i_1}\ldots
s_{i_{k}}(\alpha_{i_{k+1}})\,\,\,\text{with
}\tilde{\alpha}_t=\alpha_{i_1}\,,
$$
then $I(w)=\{ \tilde{\alpha}_k\}_{k=1}^t,$ with all
$\tilde{\alpha}_k$ different (see for instance
\cite[Section 1.7]{Humphreys2}).

For $j\in \{1, \ldots, n\},$ we will denote by $W_{j}$ the parabolic
subgroup of $W$ spanned by the set of simple reflections
$\{s_i\}_{i\ne j}\,.$ Likewise, we denote by $R_j$ the set of roots
spanned by the set of simple roots $\{\alpha_{i}\}_{i\ne j}\,.$ We
denote the set of minimal length representatives of the quotient
$W/W_{j}$ by
\begin{align*}
W^{j}:
&=\{w\in W \mid I(w)\subset R^+\backslash R_j^+\}\,.
\end{align*}


\begin{theorem}
A sequence of moves of the modified Kostant game on the Dynkin
diagram of coroots of $\Gamma$ at the vertex $j$ encodes the reduced
expression of some element in $W^{j}\,.$ Conversely, any reduced
expression of an element in $W^j$ can be obtained in this way.
\end{theorem}
\begin{proof}
We encode a sequence of moves of the modified Kostant game with an
ordered arrangement of integers which indicates the position of the
vertices where we place the chips every time that we play the
modified Kostant game. So for instance the arrangement $(i_1, i_2,
\cdots, i_t)$ indicates that in our first move we place chips on the
$i_1$-vertex, then in our second move we place chips on the
$i_2$-vertex, and we end by placing chips on the $i_t$-vertex. Note
that at every step we have not encoded the number of chips that we
place on each vertex but we explain now how to obtain this number.

Given an arrangement $(i_1, i_2, \cdots, i_t)$ that codifies a
sequence of moves of the modified Kostant game, we define a sequence
$\{w_1, \ldots, w_t\}$ of elements in the Weyl group by
$$
w_l:=s_{i_l}\ldots s_{i_2}s_{i_1}
$$
and a set of roots $\{\tilde{\alpha}_1, \tilde{\alpha}_2, \ldots,
\tilde{\alpha}_t\}$ by

\begin{equation}\label{roots}
\tilde{\alpha}_{l}=w_l^{-1}(\alpha_{i_{l+1}})=s_{i_1} s_{i_2}\ldots
s_{i_l}(\alpha_{i_{l+1}})\,
\end{equation}
for $1\leq l < t$ and $\tilde{\alpha}_t=\alpha_j\,.$ Note that
$w_1=s_j\,.$

We write
$$
\tilde{\alpha}_{l}=k^{l}\alpha_j+\ldots\,
$$
as a linear combination of simple roots. We claim that $k^l$ is the
number of chips that we locate at the $i_{l+1}$-vertex  when we play
the modified Kostant game.

We construct an auxiliary vector space by adding an element
$\tilde{\beta}$ to the basis of simple coroots $\{\check{\alpha}_1,
\ldots, \check{\alpha}_n\}$ and extend the invariant inner product
defined on $\mathfrak{t}$ to a bilinear product on
$\mathfrak{t}\oplus \langle \tilde{\beta}\rangle$ by saying that
$$
(\tilde{\beta},\tilde{\beta})=2 \,\,,\,\, (\alpha_i,
\tilde{\beta})=-\delta_{ij}\,
$$
for $j=1, \ldots, n\,.$ We also extend the action of the Weyl group
$W$ defined on $\mathfrak{t}$ to $\mathfrak{t}\oplus \langle
\tilde{\beta}\rangle\,$ by
$$
s_{\alpha_i}(\tilde{\beta})=\tilde{\beta}-(\alpha_i,
\tilde{\beta})\check{\alpha_i}=\tilde{\beta}+\delta_{ij}\check{\alpha}_j\,
$$
for every simple root $\alpha_i\,.$ The bilinear product that we
define for $\mathfrak{t}\oplus \langle \tilde{\beta}\rangle$ is
invariant with respect to the action of $W$ that we just defined on
it.

The elements
$$
w_{l+1}(\tilde{\beta}):=s_{i_{l+1}}(w_l(\tilde{\beta}))=w_l(\tilde{\beta})-(\alpha_{i_{l+1}},
w_l(\tilde{\beta}))\check{\alpha}_{i_{l+1}}\,
$$
follow the same recursion formula defined by the sequence of moves
$(i_1, i_2, \ldots, i_t)$ of the modified Kostant game at the vertex
$j,$ i.e., if
$$
w_l(\tilde{\beta})=\sum_{i=1}^nc_i\check{\alpha}_i+\tilde{\beta},
$$
and $N(i_{l+1})$ is the set of vertices in the Dynkin diagram that
are adjacent to $i_{l+1}$ and $\check{n}_{i, j}:=-(\alpha_i,
\check{\alpha}_j)$ is the number of arrows coming to $i$ from $j$ in
the Dynkin diagram of simple coroots (that is the same as the number
of arrows $n_{j, i}$ from $i$ to $j$ in the Dynkin diagram of simple
roots), then
\begin{align*}
w_{l+1}(\tilde{\beta})&=\sum_{i=1}^nc_i\check{\alpha}_i+\tilde{\beta}-\Bigl(\alpha_{i_{l+1}}, \sum_{i=1}^nc_i\check{\alpha}_i+\tilde{\beta}\Bigr)\check{\alpha}_{i_{l+1}}\\
&=\sum_{i\ne
i_{l+1}}c_i\check{\alpha}_i+\tilde{\beta}+\Bigl(-c_{i_{l+1}}-\sum_{k\in
N(i_{l+1})} c_{k}(\alpha_{i_{l+1}},
\check{\alpha}_{k})+\delta_{ji_{l+1}}\Bigr)\check{\alpha}_{i_{l+1}}\\
&=\sum_{i\ne
i_{l+1}}c_i\check{\alpha}_i+\tilde{\beta}+\Bigl(-c_{i_{l+1}}+\sum_{k\in
N(i_{l+1})}\check{n}_{i_{l+1},k}
c_{k}+\delta_{ji_{l+1}}\Bigr)\check{\alpha}_{i_{l+1}}\,.
\end{align*}
As the bilinear form that we define for $\mathfrak{t}\oplus \langle
\tilde{\beta}\rangle$ is invariant with respect to the action of
$W,$ we get
$$
k^l=-(\tilde{\beta}, \tilde{\alpha}_l)=-(\tilde{\beta},
w_l^{-1}(\alpha_{i_{l+1}}))=-(w_l(\tilde{\beta}),
\alpha_{i_{l+1}})=\dfrac{w_{l+1}(\tilde{\beta})-w_{l}(\tilde{\beta})}{\check{\alpha}_{i_{l+1}}}
$$
and the claim follows.

The Theorem now follows from the fact that $k^l>0$ for all $1\leq
l\leq t$ if and only if $\{\tilde{\alpha}_{1}, \ldots,
\tilde{\alpha}_t\}\subset R^+\backslash R^{+}_{j}\,.$ The later is
equivalent to say that $w_l\in W^{j}$ and $w_l=s_{i_l}\ldots
s_{i_2}s_{i_1}$ is a reduced expression for all $1\leq l \leq t\,.$

\end{proof}

Now we are ready to give the proofs of Theorem \ref{Main theorem1}
and Theorem \ref{Main theorem2}:

\begin{proof}[Proof of Theorem \ref{Main theorem1} and Theorem \ref{Main theorem2}]

We keep the notation of the proof of the previous Theorem. We show
first that the modified Kostant game at the vertex $j$ terminates.
Let $w_0$ be the longest element in $W^{j},$ i.e, the unique element
in $W^{j}$ that sends all positive roots in $R^+\smallsetminus
R^+_{j}$ to negative roots.

Let $(i_1, i_2, \ldots, i_t)$ be an arrangement that encodes a
sequence of moves of the modified Kostant game. The previous Theorem
implies that
$$
s_{i_t}\ldots s_{i_2}s_{i_1}
$$
is a reduced expression of an element in $W^{j}\,.$ We can always
complete the reduced expression into a reduced expression of $w_0,$
i.e., there exist simple roots $\alpha_{i_{t+1}}, \ldots,
\alpha_{i_m}$ such that
$$
w_0=s_{i_m}\ldots s_{i_{t+1}} s_{i_t}\ldots s_{i_2} s_{i_1}
$$
is a reduced expression.

The previous statement implies that the arrangement
$$
(i_1, i_2, \ldots, i_t, i_{t+1}, \ldots, i_m)
$$
encodes a sequence of moves of the modified Kostant game. In
particular, the modified Kostant game terminates. To end the proof
of Theorem \ref{Main theorem1}, we need to show that the game ends
in a unique terminating state. But this follow from the fact the
longest element in $W^j$ is unique.


We are done with the proof of Theorem \ref{Main theorem1} and now we
continue with the proof of Theorem \ref{Main theorem2}. Let us
assume now that
$$
w_0=s_{i_m}\cdot\ldots\cdot s_{i_2} s_{i_1}
$$
is a reduced expression, and let $\{\tilde{\alpha}_1,
\tilde{\alpha}_2, \ldots, \tilde{\alpha}_m\}$ be the set of roots
defined in Equation \ref{roots}. We know that this sets equals to
$R^+\backslash R^+_{j}$ and also from the proof of the previous
Proposition that if we write
$$
\tilde{\alpha}_l=k^l\alpha_j+\ldots,
$$
then $k^l$ is the number of chips that we locate at the $i_l$-vertex
when we play the Kostant game. Thus
$$
h_j=\sum_l k^l= -\sum_{l}(\tilde{\alpha}_l, \tilde{\beta})=
-\sum_{\alpha\in R^+\backslash R^+_{j}}(\alpha,
\tilde{\beta})=-\sum_{\alpha\in R^+}(\alpha, \tilde{\beta}),
$$
and we are done.

\end{proof}

We illustrate the overall idea of the proof of Theorem \ref{roots of
hil} with the following Example.

\begin{example}
Let $S=\{\alpha_1, \alpha_2, \alpha_3, \alpha_4\}$ be the set of
simple roots of $F_4$ with Dynkin diagram shown below:
\begin{figure}[!ht]
\centering
\begin{tikzpicture}[scale=1.4,transform shape]
  \node[draw,shape=circle,scale=0.7] (a) at (0,0) {};
  \node[draw,shape=circle,scale=0.7] (d) at (3,0) {};
  \draw (1,-0.06) -- (2,-0.06);
  \draw (1,0.06) -- (2,0.06);
 \node[draw,fill=white,shape=circle,scale=0.7] (b) at (1,0) {};
  \node[draw,fill=white,shape=circle,scale=0.7] (c) at (2,0) {};
   \draw(a) -- (b)
        (c) -- (d);
 \node [scale=0.7,below] at (0,-0.2) {$\alpha_1$};
 \node [scale=0.7,below] at (1,-0.2) {$\alpha_2$};
 \node [scale=0.7,below] at (2,-0.2) {$\alpha_3$};
 \node [scale=0.7,below] at (3,-0.2) {$\alpha_4$};
 \draw (1.42,0.16) -- (1.58,0) -- (1.42,-0.16);
\end{tikzpicture}
\end{figure}
Assume that $S_P=\{\alpha_1, \alpha_2, \alpha_3\}\,.$

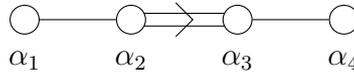
\begin{enumerate}

\item First we select the simple roots adjacent to $S_P$ in the Dynkin
diagram. In this case the only one is $\beta=\alpha_4$ and it is
adjacent to $\alpha_3$ in $S_P\,.$

\item We play the modified Kostant game on the Dynkin diagram of
coroots $\{\check{\alpha}_1, \check{\alpha}_2, \check{\alpha}_3\}$
at the $3$-vertex until it terminates to produce a string of coroots
for $P$ and $\beta$:

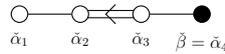
\begin{figure}[!ht]
\centering
\begin{tikzpicture}[scale=0.8,transform shape]
  \node[draw,shape=circle,scale=0.7] (a) at (0,0) {};
  \node[draw,fill=black,shape=circle,scale=0.7] (d) at (3,0) {};
  \draw (1,-0.06) -- (2,-0.06);
  \draw (1,0.06) -- (2,0.06);
 \node[draw,fill=white,shape=circle,scale=0.7] (b) at (1,0) {};
  \node[draw,fill=white,shape=circle,scale=0.7] (c) at (2,0) {};
   \draw(a) -- (b)
        (c) -- (d);
 \node [scale=0.7,below] at (0,-0.2) {$\check{\alpha}_1$};
 \node [scale=0.7,below] at (1,-0.2) {$\check{\alpha}_2$};
 \node [scale=0.7,below] at (2,-0.2) {$\check{\alpha}_3$};
 \node [scale=0.7,below] at (3,-0.2) {$\check{\beta}=\check{\alpha}_4$};
 \draw (1.58,0.16) -- (1.42,0) -- (1.58,-0.16);
\end{tikzpicture}
\end{figure}


\begin{figure}[!ht]
\centering
\begin{tikzpicture}[scale=0.8,transform shape]
  \node[draw,shape=circle,scale=0.5] (a) at (0,0) {\phantom{$0$}};
  \node[draw,fill=white,shape=circle,scale=0.5] (b) at (1,0) {\phantom{$0$}};
  \node[draw,fill=white,shape=circle,scale=0.5] (c) at (2,0) {\phantom{$0$}};
  \node[draw,fill=black,shape=circle,scale=0.5] (d) at (3,0) {$\color{white}{1}$};
  \draw(a) -- (b)
       (c) -- (d);
   \draw (1.18,-0.06) -- (1.82,-0.06);
   \draw (1.18,0.06) -- (1.82,0.06);
   \draw (1.58,0.16) -- (1.42,0) -- (1.58,-0.16);
   \draw [->](1.5,-0.25) -- (1.5,-0.75);


 \node[draw,shape=circle,scale=0.5] (a) at (0,-1.0) {\phantom{$0$}};
 \node[draw,fill=white,shape=circle,scale=0.5] (b) at (1,-1.0) {\phantom{$0$}};
 \node[draw,fill=white,shape=circle,scale=0.5] (c) at (2,-1.0) {$1$};
 \node[draw,fill=black,shape=circle,scale=0.5] (d) at (3,-1.0) {$\color{white}{1}$};
 \draw(a) -- (b)
      (c) -- (d);
 \draw (1.18,-1.06) -- (1.82,-1.06);
 \draw (1.18,-0.94) -- (1.82,-0.94);
 \draw (1.58,-0.84) -- (1.42,-1) -- (1.58,-1.16);
 \draw [->](1.5,-1.25) -- (1.5,-1.75);


 \node[draw,shape=circle,scale=0.5] (a) at (0,-2.0) {\phantom{$0$}};
 \node[draw,fill=white,shape=circle,scale=0.5] (b) at (1,-2.0) {$2$};
 \node[draw,fill=white,shape=circle,scale=0.5] (c) at (2,-2.0) {$1$};
 \node[draw,fill=black,shape=circle,scale=0.5] (d) at (3,-2.0) {$\color{white}{1}$};
  \draw (1.18,-2.06) -- (1.82,-2.06);
  \draw (1.18,-1.94) -- (1.82,-1.94);
   \draw(a) -- (b)
        (c) -- (d);
  \draw (1.58,-1.84) -- (1.42,-2) -- (1.58,-2.16);

  \draw [->](0.45,-2.25) -- (-0.45,-3.25);
  \draw [->](2.55,-2.25) -- (3.45,-3.25);


 \node[draw,shape=circle,scale=0.5] (a) at (-3,-3.5) {$2$};
 \node[draw,fill=white,shape=circle,scale=0.5] (b) at (-2,-3.5) {$2$};
 \node[draw,fill=white,shape=circle,scale=0.5] (c) at (-1,-3.5) {$1$};
 \node[draw,fill=black,shape=circle,scale=0.5] (d) at (0,-3.5) {$\color{white}{1}$};
  \draw (-1.82,-3.56) -- (-1.18,-3.56);
  \draw (-1.82,-3.44) -- (-1.18,-3.44);
   \draw(a) -- (b)
        (c) -- (d);
  \draw (-1.42,-3.34) -- (-1.58,-3.5) -- (-1.42,-3.66);

 \node[draw,shape=circle,scale=0.5] (a) at (3,-3.5) {\phantom{$0$}};
 \node[draw,fill=white,shape=circle,scale=0.5] (b) at (4,-3.5) {$2$};
  \node[draw,fill=white,shape=circle,scale=0.5] (c) at (5,-3.5) {$2$};
  \node[draw,fill=black,shape=circle,scale=0.5] (d) at (6,-3.5) {$\color{white}{1}$};
  \draw (4.18,-3.56) -- (4.82,-3.56);
  \draw (4.18,-3.44) -- (4.82,-3.44);
   \draw(a) -- (b)
        (c) -- (d);
  \draw (4.58,-3.34) -- (4.42,-3.5) -- (4.58,-3.66);

  \draw [->](-0.45,-3.75) -- (0.45,-4.75);
  \draw [->](3.45,-3.75) -- (2.55,-4.75);


  \node[draw,shape=circle,scale=0.5] (a) at (0,-5) {$2$};
  \node[draw,fill=white,shape=circle,scale=0.5] (b) at (1,-5) {$2$};
  \node[draw,fill=white,shape=circle,scale=0.5] (c) at (2,-5) {$2$};
  \node[draw,fill=black,shape=circle,scale=0.5] (d) at (3,-5) {$\color{white}{1}$};
  \draw (1.18,-5.06) -- (1.82,-5.06);
  \draw (1.18,-4.94) -- (1.82,-4.94);
  \draw(a) -- (b)
       (c) -- (d);
 \draw (1.58,-4.84) -- (1.42,-5) -- (1.58,-5.16);
 \draw [->](1.5,-5.25) -- (1.5,-5.75);


 \node[draw,shape=circle,scale=0.5] (a) at (0,-6) {$2$};
 \node[draw,fill=white,shape=circle,scale=0.5] (b) at (1,-6) {$4$};
 \node[draw,fill=white,shape=circle,scale=0.5] (c) at (2,-6) {$2$};
 \node[draw,fill=black,shape=circle,scale=0.5] (d) at (3,-6) {$\color{white}{1}$};
  \draw (1.18,-6.06) -- (1.82,-6.06);
  \draw (1.18,-5.94) -- (1.82,-5.94);
   \draw(a) -- (b)
        (c) -- (d);
 \draw (1.58,-5.84) -- (1.42,-6) -- (1.58,-6.16);
 \draw [->](1.5,-6.25) -- (1.5,-6.75);


 \node[draw,shape=circle,scale=0.5] (a) at (0,-7) {$2$};
 \node[draw,fill=white,shape=circle,scale=0.5] (b) at (1,-7) {$4$};
 \node[draw,fill=white,shape=circle,scale=0.5] (c) at (2,-7) {$3$};
 \node[draw,fill=black,shape=circle,scale=0.5] (d) at (3,-7) {$\color{white}{1}$};
  \draw (1.18,-7.06) -- (1.82,-7.06);
  \draw (1.18,-6.94) -- (1.82,-6.94);
   \draw(a) -- (b)
        (c) -- (d);
 \draw (1.58,-6.84) -- (1.42,-7) -- (1.58,-7.16);

\end{tikzpicture}
\end{figure}

In this case we obtain two arrangements that encode the moves of the
Kostant game until it ends: $(3, 2, 1, 3, 2, 3)$ and $(3, 2, 3, 1,
2, 3)\,.$ The corresponding strings of coroots are $\{
\check{\beta}, \check{\alpha}_3+\check{\beta},
2\check{\alpha}_2+\check{\alpha}_3+\check{\beta},
2\check{\alpha}_1+2\check{\alpha}_2+\check{\alpha}_3+ \check{\beta},
2\check{\alpha}_1+2\check{\alpha}_2+2\check{\alpha}_3+
\check{\beta},
2\check{\alpha}_1+4\check{\alpha}_2+2\check{\alpha}_3+
\check{\beta},
2\check{\alpha}_1+4\check{\alpha}_2+3\check{\alpha}_3+
\check{\beta}\}$ and $\{ \check{\beta},
\check{\alpha}_3+\check{\beta},
2\check{\alpha}_2+\check{\alpha}_3+\check{\beta},
2\check{\alpha}_2+2\check{\alpha}_3+ \check{\beta},
2\check{\alpha}_1+2\check{\alpha}_2+2\check{\alpha}_3+
\check{\beta},
2\check{\alpha}_1+4\check{\alpha}_2+2\check{\alpha}_3+
\check{\beta},
2\check{\alpha}_1+4\check{\alpha}_2+3\check{\alpha}_3+
\check{\beta}\},$ respectively.

\item It is possible that the strings have gaps, but whenever there
is a gap between two coroots, we can always fill the gaps. For
instance, we fill the gaps for the string of coroots encoded by $(3,
2, 3, 1, 2, 3)$ and obtain the following good string $\{
\check{\beta},
\check{\alpha}_3+\check{\beta},\check{\alpha}_2+\check{\alpha}_3+\check{\beta},
2\check{\alpha}_2+\check{\alpha}_3+\check{\beta},
2\check{\alpha}_2+2\check{\alpha}_3+ \check{\beta},
\check{\alpha}_1+2\check{\alpha}_2+2\check{\alpha}_3+ \check{\beta},
2\check{\alpha}_1+2\check{\alpha}_2+2\check{\alpha}_3+
\check{\beta},
2\check{\alpha}_1+3\check{\alpha}_2+2\check{\alpha}_3+
\check{\beta},
2\check{\alpha}_1+4\check{\alpha}_2+2\check{\alpha}_3+
\check{\beta},
2\check{\alpha}_1+4\check{\alpha}_2+3\check{\alpha}_3+
\check{\beta}\}.$

\item The sequence of moves of the modified Kostant game at the 3-vertex
gives a reduced expression of the longest element of the set of
minimal length representatives of the quotient $W_{\alpha_1,
\alpha_2, \alpha_3}/W_{\alpha_1, \alpha_2}\,.$ With this reduced
expression, we obtain all the positive roots spanned by $\{\alpha_1,
\alpha_2, \alpha_3\}$ that pair non-trivial with $\beta, $ i.e., the
roots of the form
$$
k_1\alpha_1+k_2\alpha_2+k_3\alpha_3
$$
where $k_1, k_2, k_3$ are non-negative integers and $k_3> 0\,.$ For
instance, the arrangement of moves $(3, 2, 3, 1, 2, 3)$ gives us the
reduced expression
$$
s_{3}s_{2}s_{1}s_{3}s_{2}s_{3}
$$
and we obtain the following roots
\begin{itemize}
\item $\tilde{\alpha}_1=\alpha_3$
\item $\tilde{\alpha}_2=s_{3}(\alpha_2)=\alpha_2+2\alpha_3$
\item $\tilde{\alpha}_3=s_{3}s_{2}(\alpha_3)=\alpha_2+\alpha_3$
\item $\tilde{\alpha}_4=s_{3}s_{2}s_{3}(\alpha_1)=\alpha_1+\alpha_2+2\alpha_3$
\item $\tilde{\alpha}_5=s_{3}s_{2}s_{3}s_{1}(\alpha_2)=\alpha_1+2\alpha_2+2\alpha_3$
\item $\tilde{\alpha}_6=s_{3}s_{2}s_{3}s_{1}s_{2}(\alpha_3)=\alpha_1+\alpha_2+\alpha_3$
\end{itemize}
Note that the coefficient $k^l$ of $\alpha_3$ in $\tilde{\alpha}_l$
equals the number of chips that we place at the Dynkin diagram every
time we play the Kostant game. The resulting good string is maximal
because
\begin{align*}
\operatorname{ht}(2\check{\alpha}_1+4\check{\alpha}_2&+3\check{\alpha}_3+
\check{\beta})=\sum_{l=1}^6 k^l+1\\&=-\sum_{\alpha\in R^+_P}(\alpha,
\check{\beta})+1=10\,
\end{align*}
and the Hilbert polynomial $H_{P}(k_4)$ is divided by
$\prod_{l=1}^{10}(k_4+l)\,.$

\end{enumerate}
\end{example}

\end{document}